\documentclass[12pt,reqno]{amsart}
\usepackage{amsmath,amssymb,amsthm,mathtools,enumitem}
\usepackage{url,needspace,fullpage}

\newtheorem{theorem}{Theorem}[section]
\newtheorem{definition}[theorem]{Definition}

\theoremstyle{plain}
\newtheorem{lemma}[theorem]{Lemma}
\newtheorem{corollary}[theorem]{Corollary}
\newtheorem{proposition}[theorem]{Proposition}

\usepackage[
  backend=biber,
  style=numeric,
  citestyle=numeric-comp,
  sorting=nty,
  giveninits=true,
  maxbibnames=99,
  isbn=false,
  doi=true,
  url=true,
  eprint=true
]{biblatex}

\DeclareFieldFormat{labelnumberwidth}{\mkbibbrackets{#1}}

\DeclareNameAlias{author}{family-given}
\DeclareNameAlias{editor}{family-given}

\DeclareDelimFormat[bib]{multinamedelim}{\addcomma\space}
\DeclareDelimFormat[bib]{finalnamedelim}{\addcomma\space}

\DeclareFieldFormat*{title}{#1}
\DeclareFieldFormat{journaltitle}{\mkbibemph{#1}}
\DeclareFieldFormat[article]{volume}{\mkbibbold{#1}}
\DeclareFieldFormat[article]{number}{#1}
\DeclareFieldFormat[article]{eid}{#1}
\DeclareFieldFormat{pages}{#1}

\renewbibmacro*{in:}{%
  \ifentrytype{article}
    {}
    {\printtext{\bibstring{in}\intitlepunct}}}

\renewbibmacro*{volume+number+eid}{%
  \printfield{volume}%
  \iffieldundef{number}
    {}
    {\printtext[parens]{\printfield{number}}}%
  \setunit{\addcomma\space}%
  \printfield{eid}}

\DeclareFieldFormat{doi}{\url{https://doi.org/#1}}

\AtEveryBibitem{%
  \clearfield{issn}%
  \iffieldundef{doi}
    {%
      \iffieldundef{eprint}
        {}
        {\clearfield{url}}%
    }
    {\clearfield{url}}%
}

\AtBeginBibliography{\emergencystretch=1em}

\usepackage[colorlinks=true,
            linkcolor=blue,
            citecolor=blue,
            urlcolor=blue,
            bookmarks=true,
            breaklinks=true]{hyperref}
\hypersetup{pdftitle={Sharp singularity asymptotic for discrete random matrices},
            pdfauthor={Xuanang Hu, Zeyan Song, Xinglong Wu}}

\numberwithin{equation}{section}
\allowdisplaybreaks[4]

\newcommand{\R}{\mathbb R}
\newcommand{\Z}{\mathbb Z}
\newcommand{\T}{\mathbb T}
\newcommand{\PP}{\mathbb P}
\newcommand{\EE}{\mathbb E}
\newcommand{\one}{\mathbf 1}
\newcommand{\spn}{\operatorname{span}}
\newcommand{\rank}{\operatorname{rank}}
\newcommand{\dist}{\operatorname{dist}}
\newcommand{\vol}{\operatorname{vol}}
\newcommand{\Comp}{\operatorname{Comp}}
\newcommand{\Incomp}{\operatorname{Incomp}}

\newcommand{\Lc}{\mathcal L}
\newcommand{\se}[1]{\mathcal S(#1)}
\newcommand{\norm}[1]{\lVert#1\rVert}

\title[Sharp singularity asymptotics]{Sharp singularity asymptotics for discrete random matrices}

\author{Xuanang Hu}
\address{Shandong University, Jinan, 250100, China.}
\email{xuananghu7@gmail.com}

\author{Zeyan Song}
\address{Shandong University, Jinan, 250100, China.}
\email{zeyansong8@gmail.com}

\author{Xinglong Wu}
\address{Shandong University, Jinan, 250100, China.}
\email{wuxinglong3@gmail.com}

\date{}

\begin{document}

\begin{abstract}
We resolve the Rademacher singularity conjecture: an $n\times n$
matrix with independent uniform $\{-1,1\}$ entries is singular
with probability $(2+o(1))n^2 2^{-n}$.
More generally, for an $n\times n$ matrix $M_n$ with independent
entries uniform on a fixed set $S\subset\R$ of cardinality $q\ge2$,
we show that
\begin{align*}
\PP\left(s_n(M_n)\le z/\sqrt n\right)
 \le Cz+a_n(S)+C'n^{1+\varepsilon}q^{-n}
\end{align*}
for all $n\ge1$, $z\ge0$, and $\varepsilon>0$, with $C=C(S)$
and $C'=C'(S,\varepsilon)$.
Here $s_n$ denotes the least singular value, and $a_n(S)$ is
the sum of the probabilities of the events involving zero rows
or columns and equal or opposite pairs of rows or columns.
\end{abstract}
\maketitle

\section{Introduction}

Let $B_n$ be an $n\times n$ random matrix whose entries are independent
and uniform on $\{-1,1\}$. A natural question is to determine the
probability that $B_n$ is singular. Two equal or opposite rows, or
two equal or opposite columns, force singularity. The probability
of their union is $(2+o(1))n^22^{-n}$, and it has long been conjectured
that these elementary dependencies give the full asymptotic:
\begin{align*}
\PP\left(B_n\text{ is singular}\right)=(2+o(1))n^22^{-n}.
\end{align*}
See Conjecture~1.1 in \cite{TV}, Conjecture~2.2 in
\cite{VuSurvey}, and equation~(4.1) in \cite{SahasrabudheICM}.

Koml\'os \cite{Komlos} proved that the singularity probability tends
to zero, and Kahn, Koml\'os and Szemer\'edi \cite{KKS} obtained the
first exponential upper bound. Subsequent work introduced several
methods that now play a central role in the subject. Tao and Vu
\cite{TVinv} developed inverse Littlewood--Offord theory, relating
large concentration probabilities to additive structure. Sharper
structural descriptions were obtained in \cite{TVsharp,NV};
see \cite{NVsurvey} for a survey of the theory and its applications. A geometric approach to invertibility through small ball
estimates and the decomposition of the unit sphere into
compressible and incompressible vectors was developed by
Rudelson \cite{Rudelson} and by Rudelson and Vershynin \cite{RV}.
Meanwhile, the exponential bound for singularity was improved by Tao and Vu
\cite{TV06,TV} and Bourgain, Vu and Wood \cite{BVW}. Tikhomirov
\cite{Tik} introduced inversion of randomness and proved the
optimal exponential estimate
\begin{align*}
\PP\left(B_n\text{ is singular}\right)=\left(\frac{1}{2}+o_{n}(1)\right)^n.
\end{align*}

Jain, Sah and Sawhney \cite{JSS} determined the singularity
asymptotic for every fixed nonconstant finitely supported real
distribution that is not uniform on its support. Their work also
gives the optimal exponential rate for uniform distributions on
finite subsets of $\R$. For sharp results on sparse Bernoulli
matrices, see \cite{LitvakTikhomirovSparse,HuangSparse}; for highly biased discrete distributions, see \cite{SongBiased}.
The precise singularity asymptotic for uniform distributions
on fixed finite sets, including the Rademacher distribution,
remained open.

We establish this asymptotic for all such distributions, with an error smaller than the main term by a polynomial factor. We also control the lower tail of the least singular value.
Let $\xi$ be uniform on a fixed finite set $S\subset\R$ with
$q=|S|\ge2$, and let $M_n=M_n(\xi)$ have independent entries
distributed as $\xi$. Write
\begin{align*}
s_n(M_n)=\inf_{x\in\mathbb S^{n-1}}\norm{M_nx}_2.
\end{align*}
For an independent copy $\xi'$ of $\xi$, define
\begin{align}
a_n(S)=2n\PP\left(\xi=0\right)^n
 +n(n-1)\left(q^{-n}+\PP\left(\xi=-\xi'\right)^n\right).
 \label{eq:an}
\end{align}
This is the sum of the probabilities of the elementary
events given by zero rows or columns and equal or opposite
pairs of rows or columns. Here $\PP\left(\xi=\xi'\right)=q^{-1}$. In particular,
$a_n(S)=\Theta_S(n^2q^{-n})$.

\begin{theorem}\label{thm:main}
Let $S\subset\R$ be a fixed set of $q\ge2$ points. There is a
constant $C_{\ref{thm:main}}>0$ depending only on $S$ such that
for every $\varepsilon>0$ there is a constant
$C'_{\ref{thm:main}}>0$ depending only on $S,\varepsilon$ for which
\begin{align}
\PP\left(s_n(M_n)\le z/\sqrt n\right)
 \le C_{\ref{thm:main}}z+a_n(S)
       +C'_{\ref{thm:main}}n^{1+\varepsilon}q^{-n}
 \label{eq:main}
\end{align}
for every $n\ge1$ and $z\ge0$.
\end{theorem}

Setting \(z=0\) yields the following estimate for the singularity probability.
\begin{corollary}\label{cor:singularity}
For every fixed $0<\varepsilon<1$,
\begin{align*}
\PP\left(M_n\text{ is singular}\right)
 =a_n(S)+O_{S,\varepsilon}(n^{1+\varepsilon}q^{-n})
 =(1+o(1))a_n(S).
\end{align*}
In particular, for independent uniform $\{-1,1\}$ entries,
\begin{align*}
\PP\left(B_n\text{ is singular}\right)
 =(2+o_{n}(1))n^{2}2^{-n}.
\end{align*}
For independent uniform $\{0,1\}$ entries, the corresponding formula is
\begin{align*}
\PP\left(M_n\text{ is singular}\right)
 =(1+o_{n}(1))n^{2}2^{-n}.
\end{align*}
\end{corollary}

Thus Corollary~\ref{cor:singularity} proves the Rademacher
singularity conjecture stated as Conjecture~1.1 in \cite{TV}.
Together with the result of Jain, Sah and Sawhney \cite{JSS}
for nonuniform distributions, it establishes the singularity conjecture for every fixed nonconstant finitely
supported real distribution.

Xie~\cite{Xie} determined the asymptotic probability that $r$
independent uniform random vectors in $\{-1,1\}^n$ are linearly
dependent for $1\le r\le n-1$. Corollary~\ref{cor:singularity}
settles the remaining case $r=n$, so together these results
determine the asymptotic probability of linear dependence
throughout the range $1\le r\le n$.

For the least singular value, Edelman \cite{Edelman} identified $n^{-1/2}$ as the natural scale
of the least singular value for matrices with independent standard
Gaussian entries. Following the work of Rudelson \cite{Rudelson},
Rudelson and Vershynin \cite{RV,RVrect} established the same scale
for independent identically distributed subgaussian entries with mean zero
and variance one,
with a lower tail bound of the form $C\varepsilon+e^{-cn}$ at
radius $\varepsilon n^{-1/2}$. In recent work, Sah,
Sahasrabudhe and Sawhney \cite{SSS} refined the coefficient of the
linear term to $1+o(1)$.

For the uniform distributions considered here,
Theorem~\ref{thm:main} gives a bound with the natural
linear dependence on $z$ and a coefficient depending
only on $S$.
It also identifies the singularity contribution
as $(1+o(1))a_n(S)$.

The main new ingredient is a Fourier averaging estimate in the inversion of randomness argument, valid up to coefficient scales a polynomial factor below $q^n$. To amplify the weak additive structure, we take moments whose order grows with $n$ and average over the coefficients. We classify the resulting sets of approximate linear relations by dimension and density. Refined counting estimates, combined with bounds on the corresponding frequency regions, yield the polynomial precision needed for the sharp singularity asymptotic.

\subsection{Notation}

All spans are over $\R$, and all logarithms are natural.
We write $[m]=\{1,\ldots,m\}$, $\one_m=(1,\ldots,1)$,
and $t_+=\max\{t,0\}$ for $t\in\R$.
The vectors $e_i$ denote the standard coordinate vectors.
We use $\|\cdot\|_1$, $\|\cdot\|_2$, and $\|\cdot\|_\infty$
for the usual $\ell_1$, $\ell_2$, and $\ell_\infty$ norms,
and $\langle\cdot,\cdot\rangle$ for the Euclidean inner product.
The Euclidean unit sphere is
$\mathbb S^{m-1}=\{x\in\R^m:\|x\|_2=1\}$.
For $I\subset[m]$, we write $x_I$ for the restriction of $x$
to the coordinates in $I$, and $I^c=[m]\setminus I$.

For a nonempty set $A\subset\R^m$, define
\begin{align*}
\dist(x,A)=\inf_{a\in A}\norm{x-a}_2,
\qquad
\dist_\infty(x,A)=\inf_{a\in A}\norm{x-a}_\infty.
\end{align*}
In particular, for a nonempty finite set $E\subset\R$,
$\dist_\infty(x,E^m)
 =\max_{i\in[m]}\min_{a\in E}|x_i-a|.$
We write $\T=\R/\Z$ and
$\norm{t}_{\T}=\dist(t,\Z)=\min_{j\in\Z}|t-j|$ for $t\in\R$.
For a linear subspace $H$, let $P_H$ denote the
orthogonal projection onto $H$.

Matrix norms without a subscript are Euclidean operator norms.
For $B=(B_{ij})\in\R^{m\times p}$, we write
\begin{align*}
\norm{B}_{\mathrm{HS}}
 =\left(\sum_{i=1}^m\sum_{j=1}^p B_{ij}^2\right)^{1/2},
\qquad
\norm{B}_{2\to\infty}
 =\max_{i\in[m]}\left(\sum_{j=1}^p B_{ij}^2\right)^{1/2}.
\end{align*}
Thus $\norm{B}_{2\to\infty}$ is the largest Euclidean row norm.
The notation $B_{I,J}$ denotes the submatrix with rows in $I$
and columns in $J$; $B_{i\cdot}$ and $B_{\cdot j}$ denote
the $i$th row and the $j$th column, respectively.
We write $B^T$ for the transpose of $B$ and $I_m$ for the
$m\times m$ identity matrix, omitting the subscript when
the dimension is clear.
For symmetric matrices $A,B$ of the same size,
$A\preceq B$ means that $B-A$ is positive semidefinite.
For a real matrix $B$, define
$\vol(B)=\det(B^TB)^{1/2}$; the empty frame has volume one.
For $B\in\R^{m\times p}$ with $p\ge1$, define
\begin{align*}
s_{\min}(B)=\inf_{x\in\mathbb S^{p-1}}\norm{Bx}_2.
\end{align*}
For an $n\times n$ matrix, we also write $s_n(B)=s_{\min}(B)$. Singular values are listed in decreasing order.

For a real random variable $X$ and $\varepsilon\ge0$, define
\begin{align*}
\Lc(X,\varepsilon)=\sup_{z\in\R}\PP\left(|X-z|\le \varepsilon\right).
\end{align*}
When multiple sources of randomness are present, subscripts on
$\PP$, $\EE$, and $\Lc$ indicate which random variables are
being averaged over, with all others held fixed.
For an event $\mathcal E$, we write $\one_{\mathcal E}$ for
its indicator.
For a bounded function $f$, we write
$\norm{f}_\infty=\sup_x|f(x)|$.

The letters $c,C$ denote positive constants whose values
may change from line to line.
Their dependence is specified locally. We use the standard notation \(f=O(g)\), \(f=\Omega(g)\), and
\(f=\Theta(g)\) for upper, lower, and two-sided bounds up to
absolute constant factors. Subscripts indicate dependence of the
implicit constants; for example,
\(O_{\alpha_1,\ldots,\alpha_k}(g)\) means that the implicit constant
depends only on \(\alpha_1,\ldots,\alpha_k\). We write
\(f=o(g)\) if \(f/g\to0\) in the specified limit.

\subsection{Organization}

Section~\ref{sec:overview} explains the proof and the difficulties
arising from uniform distributions and arbitrary real supports.
Section~\ref{sec:count} proves the counting and geometric estimates.
Section~\ref{sec:averaging} establishes the Fourier averaging theorem.
Section~\ref{sec:normals} applies it to random normal vectors and
completes the proof of the main theorem. Appendix~\ref{sec:algebraic} records the determinant estimate
available for algebraic supports.

\section{Outline of the proof}\label{sec:overview}

We explain the precision required in the uniform case and the
main ideas behind our averaging estimate. To obtain the required polynomial precision, we develop a novel
Fourier argument within the framework of inversion of randomness. Fourier analysis reduces
the problem to counting linear relations in finite product sets.
We treat separately relation sets of large and small density.

\subsection{Related work}

We follow the geometric approach of Rudelson \cite{Rudelson}
and Rudelson and Vershynin \cite{RV}. After decomposing the unit
sphere into compressible and incompressible vectors, the latter
are treated through distances between a column and the span of
the remaining columns. Conditioning on those columns leads to
the Littlewood--Offord concentration function
\begin{align*}
\rho_\varepsilon(v)
=\Lc_\xi\left(\sum_{i=1}^n\xi_i v_i,\varepsilon\right),
\end{align*}
where the $\xi_i$ are independent copies of $\xi$. The problem
is to balance the size of a net at a given concentration level
against the probability that a fixed vector is nearly orthogonal
to the independent columns.

Inverse Littlewood--Offord theory obtains such nets from
arithmetic structure. In the appropriate concentration ranges,
Tao and Vu \cite{TVinv,TVsharp} place most coefficients in a
\emph{generalized arithmetic progression} of controlled rank
and size. The \emph{essential least common denominator} of
Rudelson and Vershynin \cite{RV} measures how closely a dilate
of the coefficient vector approaches the integer lattice.

Tikhomirov \cite{Tik} introduced \emph{inversion of randomness}
to count discretized vectors without first describing their
arithmetic structure. If $X$ is uniform on a finite product set
$\mathcal A\subset\Z^n$, then
\begin{align*}
\left|\{x\in\mathcal A:\rho_\varepsilon(x)\ge\lambda\}\right|
=|\mathcal A|\,
 \PP_X\left(\rho_\varepsilon(X)\ge\lambda\right).
\end{align*}
The counting problem thus becomes an estimate for concentration
with random coefficients. In his discrete Bernoulli formulation,
concentration is controlled by the supremum of a random average
of integer translates of a nonnegative function in $\ell^1(\Z)$.
The averaging is performed successively as the coefficients are
exposed. The resulting estimates give the sharp exponential
rate, with a fixed slack: in the Bernoulli$(1/2)$ case, they
allow factors of the form
\begin{align*}
\left(\frac12+\eta\right)^n=2^{-n}(1+2\eta)^n.
\end{align*}
For fixed $\eta>0$, this loss is exponential. The dependence
of the constants and dimension thresholds on $\eta$ prevents
a direct substitution of a parameter tending to zero with $n$.

For nonuniform distributions, Jain, Sah and Sawhney \cite{JSS}
use the strict gap
\begin{align*}
e^{-H(\xi)}<\sum_a p_a^2,
\qquad p_a=\PP(\xi=a),\quad H(\xi)=-\sum_a p_a\log p_a,
\end{align*}
to make the unstructured contribution exponentially smaller
than the contribution from equal rows. For uniform distributions,
both quantities equal $q^{-1}$. Their sharp estimate for
structured vectors remains applicable and is used in
Lemma~\ref{lem:structured}.

At concentration scale $T$, the rounded coefficient vectors
have scale $N=\Theta(T^{-1})$. We prove an averaging estimate
throughout
\begin{align*}
1\le N\le q^n(n+1)^{-D}
\end{align*}
for every fixed $D>1$. The upper endpoint corresponds to an
error bound of order $(n+1)^Dq^{-n}$ after passing to random
normals. Since $a_n(S)=\Theta_S(n^2q^{-n})$, taking $1<D<2$
makes this error smaller than the main term by a polynomial
factor. The choice $D=1+\varepsilon$ gives the remainder in
Theorem~\ref{thm:main}.

\subsection{Fourier averaging at polynomial precision}

Theorem~\ref{thm:averaging} gives the following estimate.
For a fixed set $E\subset\R$ of $q\ge2$ points and fixed
$\delta\in(0,1]$, $K\ge1$, there is
$L=L_{\ref{thm:averaging}}(E,\delta,K)>0$, independent of
$M,D$, such that
\begin{align*}
\PP_X\left(
\Lc_b\left(\sum_{i=1}^n b_iX_i,\sqrt n\right)
\ge L/N\right)\le e^{-Mn}
\end{align*}
for every $M\ge1$, $D>1$ and
$n\ge n_{\ref{thm:averaging}}(E,\delta,K,M,D)$, uniformly
over integers $1\le N\le q^n(n+1)^{-D}$. Here $X$ is uniform
on any product set admissible with parameters $(N,n,K,\delta)$
as in Definition~\ref{def:admissible}, and the $b_i$ are
independent and uniform on $E$, independently of $X$.
Concentration is taken over $b$ with $X$ fixed.

We take $E$ to be a translate of $S$ containing zero, which
preserves the concentration of each linear form. Fourier
methods for concentration functions have been used extensively
\cite{Esseen,Halasz,FS}. We combine them with high moments
and averaging over the coefficients.

We first add independent uniform perturbations on $[-1/2,1/2]$
to $X$, obtaining continuous coefficients $Z$.
Gaussian smoothing then reduces the problem to bounding
$\norm{F_Z}_\infty$ by $C/(N\sqrt n)$, where
\begin{align*}
F_Z(x) =\EE_b f_n(x+\langle b,Z\rangle) =\int_\R e^{2\pi ix\theta}\omega_n(\theta)
             \prod_{i=1}^n\varphi(Z_i\theta)\,d\theta.
\end{align*}
Here $f_n$ is a suitable centered Gaussian density of variance
$\Theta_E(n)$, $\omega_n=\widehat f_n$, and
$\varphi(t)=q^{-1}\sum_{a\in E}e^{2\pi iat}$.
The main difficulty is the contribution $R_Z$ of the frequencies
$|\theta|>n^2/N$. The remaining frequencies are controlled by
admissibility and a fixed moment estimate.

To amplify weak additive structure, we use an even moment order
$k=An/\log(n+1)+O(1)$, where $A$ is fixed in terms of $E,M,D$.
A derivative bound reduces the supremum estimate to a bound on
the nonnegative spatial moment
\begin{align*}
\EE_Z\int_\R R_Z(x)^k\,dx.
\end{align*}
In this moment, we bound the coefficient densities by $2\nu_i$,
where the $\nu_i$ are smooth probability densities with Gaussian
Fourier decay at scale $N^{-1}$. We then expand the Fourier
integral before taking absolute values. Integration in $x$
restricts the frequencies to
\begin{align*}
H_k=\{\theta\in\R^k:\theta_1+\cdots+\theta_k=0\}.
\end{align*}
Averaging over the coefficients gives the kernels
\begin{align*}
K_i(\theta)=q^{-k}\sum_{v\in E^k}
                   \widehat\nu_i(-\langle v,\theta\rangle).
\end{align*}
After taking absolute values, the integrand contains
$\prod_{i=1}^n|K_i(\theta)|$. The Gaussian decay singles out
the approximate relations
\begin{align*}
V(\theta)=\{v\in E^k:|\langle v,\theta\rangle|\le R/N\},
\qquad R=\Theta_E(\sqrt k).
\end{align*}

The cardinality of $V(\theta)$ controls the kernels, while
independent vectors in $V(\theta)$ constrain the frequencies.
For each $\theta$, we select a subspace $U$ approximating
$V(\theta)$, together with an injective projection of
$V(\theta)$ onto $r=\dim U$ coordinates.
We classify frequencies by the codimension $d$ and integer
density loss $t$:
\begin{align*}
d=k-r,\qquad q^{r-t-1}<|V(\theta)|\le q^{r-t}.
\end{align*}
For a suitable choice of $R$, we have
\begin{align*}
|K_i(\theta)|\le q^{-d-t}+q^{-3k}.
\end{align*}
Thus the density loss supplies additional decay in each of
the $n$ kernel factors. We fix a large integer $T$, independently
of $n$, and treat $t\le T$ and $t>T$ separately. These correspond
to relation sets of large and small density, respectively.
In the first case we count subspaces; in the second we average
over bases and their corresponding intersections of slabs.

\subsection{Relations of large density}

For $t\le T$, the selected space satisfies $U=\spn V(\theta)$
and contains no coordinate vector. Write $u=t+1$.
Theorem~\ref{thm:count} bounds the number of possible spaces by
\begin{align*}
\exp\{d\log(k+1)+O_{q,u}(k)\},
\qquad d\ge k/2-O_{q,u}(1).
\end{align*}

The leading term $d\log(k+1)$ comes from the choices of row
supports in the graph representation
$U=\{(x,Gx):x\in\R^r\}$.
The density condition and Littlewood--Offord estimates give
a common set of $O_{q,u}(1)$ input coordinates outside which
every row of $G$ has at most one nonzero entry.
For each row, there are at most $r+1$ choices for its support
outside this set, including the empty support.
Together these choices contribute at most $(r+1)^d$;
all remaining choices contribute $\exp\{O_{q,u}(k)\}$.

We fix $U$ and integrate over the region where $U(\theta)=U$.
The constant vectors in $V(\theta)$ give $G\one_r=\one_d$.
We use the coordinates
\begin{align*}
\theta=(z-G^T\xi,\xi),
\qquad z\in H_r,\quad \xi\in\R^d.
\end{align*}
This change of variables has Jacobian one with respect to
the coordinate measures used below.
Since $\langle(x,Gx),\theta\rangle=\langle x,z\rangle$,
the relation constraints involve only $z$.
The density condition restricts $z$ to a small box, while
the Gaussian factors control integration over the free
coordinates $\xi$.

Combining the integral estimate with the kernel decay and
the normalization in the moment estimate gives a factor
$(N/q^n)^d$ that offsets the subspace count.
For $N\le q^n(n+1)^{-D}$, the relevant terms in the exponent are
\begin{align*}
d[\log(k+1)-D\log(n+1)].
\end{align*}
Since $d\ge k/2-O_{q,T}(1)$ and
$k=An/\log(n+1)+O(1)$, choosing $A$ sufficiently large
makes this negative contribution dominate the remaining
exponential factors and yields the required bound.
The coefficient one in the subspace count therefore allows
us to take any fixed $D>1$.
\subsection{Relations of small density}\label{sec:realsupports}

For $t>T$, we work with approximate relations.
The coordinate projection in Lemma~\ref{lem:relations}
maps $V(\theta)$ injectively onto a spanning subset of $E^r$.
An input basis $W\in E^{r\times r}$ and its corresponding
output columns $Y\in E^{d\times r}$ determine the
interpolation map $x\mapsto YW^{-1}x$.
For fixed $W$, Proposition~\ref{prop:approxcount} bounds
the number of matrices $Y$ for which this map sends at least
$q^{r-t-1}$ points of $E^r$ within the prescribed distance
of $E^d$ by
\begin{align*}
\exp\{O_q(k(t+2)+d\log(d+1))\}.
\end{align*}
The bound is independent of $\norm{W^{-1}}$.
Its proof recursively reduces the input dimension while
counting the original rows of $Y$.
The decrease in density loss absorbs the cost of specifying
the selected rows at each step.

To use this count in the Fourier integral, we need many input bases with controlled interpolation coefficients and determinants. For algebraic supports, the determinant estimate in Appendix~\ref{sec:algebraic} gives direct control of the inverse Jacobian. For arbitrary real supports, we use the density of the projected relation set to obtain many partial bases with controlled least singular values.
Only $O_E(t+1)$ further columns are needed to complete them.
Maximal volume completion controls the interpolation
coefficients, while comparison with the basis supplied by
Lemma~\ref{lem:relations} gives a determinant bound.
The dependence on its least singular value is confined to
these completion directions in Lemma~\ref{lem:densebases}.
The interpolation bounds ensure that the resulting pairs
$(W,Y)$ satisfy the condition in
Proposition~\ref{prop:approxcount}.

The full relation vectors corresponding to $(W,Y)$ restrict
the frequencies to an intersection of slabs.
Each frequency in the region under consideration belongs
to many such intersections.
Averaging their indicators uses this multiplicity to reduce
the cost of summing over input bases; the interpolation
count controls the choices of $Y$ for each $W$.
For each intersection, coordinates given by the relation
linear forms contribute an inverse Jacobian proportional
to $|\det W|^{-1}$, while the Gaussian factors control
integration over the free coordinates.

Combining the integral estimate with this averaged cover
and the kernel decay gives the required bound.
Once $A$ is fixed, we choose $T$ sufficiently large for
the additional decay $q^{-nt}$ to absorb the counting
and inverse Jacobian costs.
\section{Counting linear relations in finite product sets}\label{sec:count}

In this section we prove the counting estimates needed for the
Fourier argument in Section~\ref{sec:averaging}. We conclude with two auxiliary lemmas concerning projections and the choice of bases in finite product sets.

Throughout this section $E\subset\R$ is fixed, $|E|=q\ge2$,
and $0\in E$. Write $\Delta=\min_{a\ne b\in E}|a-b|>0$.

\subsection{Counting subspaces}
Let $X$ be uniformly distributed on $E^k$. For any subspace
$U\subset\R^k$, we have
\begin{align*}
\PP(X\in U)=q^{-k}|U\cap E^k|.
\end{align*}
These probabilities arise naturally in the study of singularity
and corank of discrete random matrices: the corank of a square
matrix is the codimension of the span of its columns.
They play an important role in the singularity estimates of
Kahn, Koml\'os and Szemer\'edi \cite{KKS} and Tao and Vu \cite{TV},
and in recent work on large corank by Hunter, Kwan, Sauermann
and Sawhney \cite{HKSS}.

We seek to bound the number of subspaces of a given codimension
for which this probability is large. Related counting estimates
for vectors with large concentration probabilities were obtained
by Ferber, Jain, Luh and Samotij \cite{FJLS}.
Our Fourier argument requires counting estimates for subspaces
of arbitrary codimension which arise as spaces of linear
relations among the frequencies in a Fourier moment.

Let $\mathcal H_d(k,u)$ be the logarithm of the number of
subspaces $U$ of codimension $d$ satisfying
$U=\spn(U\cap E^k)$ and $|U\cap E^k|\ge q^{k-d-u}$.

\begin{theorem}\label{thm:count}
There is a constant $C_{\ref{thm:count}}>0$ depending only on
$q$ such that for $0\le d\le k$ and $u\ge0$,
\begin{align}
\mathcal H_d(k,u)
\le d\log(k+1)+C_{\ref{thm:count}}k(u+2)^2q^{2u}.
\label{eq:refinedcount}
\end{align}
Moreover, if a subspace $U\subset\R^k$ of codimension $d$
contains no coordinate vector and satisfies
$|U\cap E^k|\ge q^{k-d-u}$, then
\begin{align}
d\ge \frac{k}{2}-C_{\ref{thm:count}}q^{2u}.
\label{eq:codim}
\end{align}
\end{theorem}

For bounded $u$, the error term in \eqref{eq:refinedcount}
is $O_{q,u}(k)$. Under the same density assumption,
\eqref{eq:codim} shows that excluding coordinate vectors
forces $d\ge k/2-O_{q,u}(1)$.

The coefficient of $d\log(k+1)$ in \eqref{eq:refinedcount}
is optimal, even when $u=0$. Indeed, for $k=2m$, each
perfect matching of $[k]$ defines a subspace by requiring
the coordinates in each pair to be equal. This subspace
has dimension $m$, contains $q^m$ points of $E^k$, and
contains no coordinate vector. Fix a nonzero element of $E$.
The vectors obtained by assigning this element to one pair
and zero to all other pairs form a basis of the subspace.
Distinct matchings give distinct subspaces, since two
distinct coordinate functionals agree on the whole subspace
if and only if their coordinates belong to the same pair.
There are $(2m)!/(2^m m!)$ perfect matchings, whose logarithm
is $m\log(2m)+O(m)=d\log(k+1)+O(k)$.

For an affine subspace $L\subset\R^k$ and any finite
$F\subset\R$, an injective coordinate projection gives
\begin{align}
|L\cap F^k|\le |F|^{\dim L}.
\label{eq:projectioncount}
\end{align}
Indeed, choose $\dim L$ coordinate functionals whose
restrictions to the direction space of $L$ are linearly
independent. The corresponding coordinate projection
is injective on $L$.
For a subspace $U$ of dimension $r=k-d$, the condition
$|U\cap E^k|\ge q^{r-u}$ therefore allows a loss of at most
a factor of $q^u$ from the upper bound $q^r$.
Since $0\in E$, coordinate subspaces attain this upper bound.

To count these subspaces, choose $r$ coordinates on which
the coordinate projection of $U$ is injective.
After permuting coordinates, write
\begin{align*}
U=\{(x,Ax):x\in\R^r\},
\qquad A\in\R^{d\times r}.
\end{align*}
The points of $U\cap E^k$ are in bijection with
\begin{align*}
\se A=\{x\in E^r:Ax\in E^d\}.
\end{align*}
In particular, $U=\spn(U\cap E^k)$ if and only if
$\se A$ spans $\R^r$.

We use the following form of the Littlewood--Offord estimate
both for exact relations and for approximate ones.
\begin{lemma}\label{lem:scalar}
Let $b_i$ be independent random variables uniformly distributed on $E$. If at least $m\ge1$ coefficients satisfy
$|a_i|>2\varepsilon/\Delta$, where $\varepsilon \ge0$, then
\begin{align*}
\Lc\left(\sum_i a_i b_i,\varepsilon \right)\le \frac {C_{\ref{lem:scalar}}} {\sqrt{m}},
\end{align*}
where $C_{\ref{lem:scalar}}$ is an absolute positive constant.
\end{lemma}
\begin{proof}
Choose a set $I$ of $m$ indices satisfying
$|a_i|\Delta>2\varepsilon$. Suppose first that $\varepsilon>0$.
For each $i\in I$, the distinct points of $a_iE$ are separated
by more than $2\varepsilon$, so
\begin{align*}
\Lc(a_ib_i,\varepsilon)=q^{-1}.
\end{align*}
Adding independent summands cannot increase concentration.
The Kolmogorov--Rogozin inequality (Theorem~1 in \cite{Rogozin}) then gives
\begin{align*}
\Lc\left(\sum_i a_ib_i,\varepsilon\right)
\le \Lc\left(\sum_{i\in I}a_ib_i,\varepsilon\right)
\le
\frac{C}{\sqrt{\sum_{i\in I}(1-\Lc(a_ib_i,\varepsilon))}}
\le \frac{C}{\sqrt m}.
\end{align*}

If $\varepsilon=0$, choose $h>0$ such that
$2h<\Delta\min_{i\in I}|a_i|$.
Applying the preceding argument at radius $h$ and using
$\Lc(X,0)\le\Lc(X,h)$ proves the same bound.
\end{proof}
Before proving Theorem~\ref{thm:count}, we count coefficient
vectors for which the corresponding linear form takes values
in $E$ on a large spanning subset of $E^p$.

\begin{lemma}\label{lem:row-count}
Let $E\subset\R$ have cardinality $q\ge2$, and define
\begin{align*}
\mathcal S(a)=\{b\in E^p:\langle a,b\rangle\in E\},
\qquad a\in\R^p.
\end{align*}
There is a constant $C_{\ref{lem:row-count}}>0$ depending
only on $q$ such that for every integer $p\ge0$ and every
$u\ge0$,
\begin{align*}
\left|
\left\{a\in\R^p:
\spn\mathcal S(a)=\R^p,\quad
|\mathcal S(a)|\ge q^{p-u}
\right\}
\right|
\le \exp\{C_{\ref{lem:row-count}}(p+1)(u+1)\}.
\end{align*}
\end{lemma}

\begin{proof}
The case $p=0$ is immediate. Assume $p\ge1$, and fix a
vector $a$ satisfying the two conditions in the statement.
Define
\begin{align*}
\Lambda=\{x/y:x,y\in E-E,\ y\ne0\}.
\end{align*}
Then $|\Lambda|\le q^4$, $0,1\in\Lambda$, and
$-\Lambda=\Lambda$.

Choose an inclusion maximal set
$I=\{i_1,\ldots,i_g\}\subseteq[p]$, with indices in increasing
order, such that the map
\begin{align*}
f:E^{g+1}&\longrightarrow\R: \quad f(x_0,\ldots,x_g)
=-x_0+\sum_{\ell=1}^g a_{i_\ell}x_\ell,
\end{align*}
is injective on $E^{g+1}$. Such a set exists, since
$I=\emptyset$ has this property. Write
$G=(a_{i_1},\ldots,a_{i_g})$.

Let $b$ and $c$ be independent and uniformly distributed
on $E^p$ and $E$, respectively.
Conditional on $(b_j)_{j\notin I}$, the equation
$\langle a,b\rangle=c$ is equivalent to
\begin{align*}
f(c,b_{i_1},\ldots,b_{i_g})
=-\sum_{j\notin I}a_jb_j.
\end{align*}
The right-hand side is a fixed real number, so injectivity
allows at most one choice of
$(c,b_{i_1},\ldots,b_{i_g})\in E^{g+1}$.
Hence
\begin{align*}
q^{-u-1}
\le q^{-p-1}|\mathcal S(a)|
=\mathbb P(\langle a,b\rangle=c)
\le q^{-g-1}.
\end{align*}
It follows that $g\le u$.

For every $j\notin I$, maximality implies that the map
obtained by adding the term $a_jx_*$ is not injective.
Subtracting two distinct inputs with the same image gives
\begin{align*}
-\delta_0+\sum_{\ell=1}^g a_{i_\ell}\delta_\ell
+a_j\delta_*=0,
\end{align*}
where all the differences belong to $E-E$.
We must have $\delta_*\ne0$, since otherwise the two inputs
would contradict injectivity on $E^{g+1}$. Therefore
\begin{align*}
a_j=\frac{\delta_0}{\delta_*}
-\sum_{\ell=1}^g a_{i_\ell}
                    \frac{\delta_\ell}{\delta_*}.
\end{align*}
The coefficients in this representation belong to $\Lambda$.
Using these representations for $j\notin I$, and taking
$P_{i_\ell}=0$ and $V_{\cdot,i_\ell}=e_\ell$ for
$\ell\in[g]$, we obtain
\begin{align*}
a=P+GV,
\qquad P\in\Lambda^p,\quad V\in\Lambda^{g\times p}.
\end{align*}
In particular, the columns of $V$ indexed by $I$ form $I_g$,
so $\operatorname{rank}V=g$.

Suppose $g\ge1$. Since $\mathcal S(a)$ spans $\R^p$ and
$V$ has rank $g$, the set $V\mathcal S(a)$ spans $\R^g$.
We may therefore choose
$z_1,\ldots,z_g\in\mathcal S(a)$ such that, for
$Z=(z_1,\ldots,z_g)$, the matrix $VZ$ is invertible.
Set $y=aZ\in E^g$. Multiplying $a=P+GV$ by $Z$ gives
\begin{align*}
y=PZ+G(VZ),
\end{align*}
and hence
\begin{align*}
G=(y-PZ)(VZ)^{-1}.
\end{align*}
Thus $a$ is uniquely determined by $P,V,Z$ and $y$. In particular, $I$ need not be specified separately.
For fixed $g\ge1$, the number of possible tuples is at most
$|\Lambda|^{p+gp}q^{pg+g}$.
When $g=0$, we have $a=P$, so the same bound holds.

Summing over $0\le g\le\min(p,\lfloor u\rfloor)$ and using
$|\Lambda|\le q^4$, we obtain
\begin{align*}
\begin{aligned}
\left|
\left\{a\in\R^p:
\spn\mathcal S(a)=\R^p,\quad
|\mathcal S(a)|\ge q^{p-u}
\right\}
\right|
&\le
\sum_{g=0}^{\min(p,\lfloor u\rfloor)}
|\Lambda|^{p+gp}q^{pg+g}\\
&\le \exp\{C(p+1)(u+1)\},
\end{aligned}
\end{align*}
where $C>0$ depends only on $q$.
\end{proof}

\begin{proof}[Proof of Theorem~\ref{thm:count}]
Throughout the proof, $c,C>0$ depend only on $q$ and may
change from line to line.

We first prove \eqref{eq:refinedcount}. Let $r=k-d$.
The cases $r=0$ and $d=0$ are immediate, so assume $r,d\ge1$.
For each counted subspace $U$, choose an injective projection
onto $r$ coordinates. After permuting coordinates, write
\begin{align*}
U=\{(x,Ax):x\in\R^r\},
\qquad A\in\R^{d\times r}.
\end{align*}
Then $\mathcal S(A)$ spans $\R^r$ and
$|\mathcal S(A)|\ge q^{r-u}$.

Let $b$ be uniformly distributed on $E^r$.
For a row $a$ of $A$ with $s\ge1$ nonzero entries,
the density assumption, Lemma~\ref{lem:scalar}, and a union
bound over $E$ give
\begin{align*}
q^{-u}
\le \mathbb P(\langle a,b\rangle\in E)
\le \frac{Cq}{\sqrt s}.
\end{align*}
Thus every row has at most $Cq^{2u}$ nonzero entries.

We claim that there is a set $J\subseteq[r]$ with $|J|\le C(u+1)q^{2u}$ such that every row has at most one nonzero entry outside $J$.

Starting with $J=\emptyset$, repeatedly choose a row with
at least two nonzero entries outside $J$ and add its support
to $J$. At termination, the required support property holds.

At each step, condition on all coordinates of $b$ except
two outside the current $J$ with nonzero coefficients
in the selected row. The elementary inequality
$|\alpha E+\beta E|\ge2q-1>q$ for $\alpha,\beta\ne0$
shows that at least one of the $q^2$ choices gives
$\langle a,b\rangle\notin E$.
All previously selected rows are supported in the current $J$.
Hence the probability that the new row condition holds,
conditional on the previous ones, is at most $1-q^{-2}$.

Since every $b\in\mathcal S(A)$ satisfies all selected
row conditions, after $h$ selections we have
\begin{align*}
q^{-u}
\le \mathbb P(\text{all selected row conditions hold})
\le (1-q^{-2})^h.
\end{align*}
It follows that $h\le C(u+1)$.
Each selected row has at most $Cq^{2u}$ nonzero entries,
which proves the claimed bound on $|J|$.

Choose $m=\lceil C(u+1)q^{2u}\rceil$ so that $|J|\le m$.
There are at most $2^k$ choices of the projection coordinates
and at most $2^r$ choices of $J$.
Each row has at most one nonzero entry outside $J$, so
the supports outside $J$ have at most $(r+1)^d$ choices.

Fix these choices. For each row $a$, let $Q$ consist of $J$
and its possible nonzero coordinate outside $J$.
Then $a$ is supported on $Q$, and $p=|Q|\le m+1$. Write $Q=\{j_1<\cdots<j_p\}$, and define
\begin{align*}
\pi_Q:\R^r\longrightarrow\R^p,
\qquad
\pi_Q(x)=(x_{j_1},\ldots,x_{j_p}).
\end{align*}
Write $a_Q=\pi_Q(a)$.
The set $\pi_Q(\mathcal S(A))$ spans $\R^p$, is contained
in $\mathcal S(a_Q)$, and satisfies
\begin{align*}
|\pi_Q(\mathcal S(A))|
\ge \frac{|\mathcal S(A)|}{q^{r-p}}
\ge q^{p-u}.
\end{align*}
Lemma~\ref{lem:row-count} therefore gives at most
$\exp\{C(p+1)(u+1)\}$ choices for $a_Q$, and hence for $a$.
Multiplying over the $d$ rows gives
\begin{align*}
\begin{aligned}
\mathcal H_d(k,u)
&\le (k+r)\log2+d\log(r+1)+Cd(m+2)(u+1)\\
&\le d\log(k+1)+Ck(u+2)^2q^{2u}.
\end{aligned}
\end{align*}
This proves \eqref{eq:refinedcount}.
The factor $(r+1)^d$ gives the coefficient one in the
leading term.

We next prove \eqref{eq:codim}.
Let $\xi_i$ be independent and uniformly distributed on $E$,
and define
\begin{align*}
\rho_s
=\sup_{\substack{
\alpha_1,\ldots,\alpha_s\in\R\setminus\{0\}\\
t\in\R}}
\mathbb P\left(\sum_{i=1}^s\alpha_i\xi_i=t\right),
\qquad s\ge1.
\end{align*}
We claim that there exists a constant $c>0$ depending only
on $q$ such that
\begin{align*}
(q\rho_s)^2
\le \left(1+c(s-2)_+\right)^{-1},
\qquad s\ge1.
\end{align*}
For $s=1,2$, the claim follows from $\rho_s\le q^{-1}$,
obtained by conditioning on all but one variable.

For $s\ge3$, we first bound the point probabilities of
a sum of three terms.
For $X=\alpha\xi_1+\beta\xi_2$, with $\alpha,\beta\ne0$,
\begin{align*}
\mathbb P(X=t)=q^{-2}|\alpha E\cap(t-\beta E)|.
\end{align*}
Equality with $q^{-1}$ requires $\alpha E=t-\beta E$,
which can hold for at most one $t$.
All other point probabilities are at most $(q-1)/q^2$.
Thus, for $\gamma\ne0$, averaging over the $q$ distinct
points $t-\gamma x$, $x\in E$, gives
\begin{align*}
\mathbb P(X+\gamma\xi_3=t)
\le \frac1q\left(\frac1q+(q-1)\frac{q-1}{q^2}\right)
=\frac1q\left(1-\frac{q-1}{q^2}\right).
\end{align*}
Adding independent summands cannot increase the largest
point probability, so the same upper bound holds for
$\rho_s$ whenever $s\ge3$.
Combining this bound with $\rho_s\le C/\sqrt s$ from
Lemma~\ref{lem:scalar} proves the claim for a sufficiently small $c>0$ depending only on $q$.

Let $U$ satisfy the assumptions of the second assertion.
The case $k=0$ is immediate, so assume $k\ge1$.
Write $U=\ker B$, where $B\in\R^{d\times k}$ has rank $d$.
Since $U$ contains no coordinate vector, $B$ has no zero
column and $d\ge1$.
After permuting coordinates and applying invertible row
operations, assume that the first $d$ columns of $B$ form $I_d$.

Group the columns by their last nonzero row.
Let $G_j$ be the set of column indices whose last nonzero
entry is in row $j$, and write $g_j=|G_j|$.
These groups partition $[k]$, and each contains a column
of $I_d$. Thus $g_j\ge1$ and $\sum_jg_j=k$.

Let $b$ be uniformly distributed on $E^k$.
The $j$-th equation of $Bb=0$ involves only
$G_j,\ldots,G_d$, with nonzero coefficients on $G_j$.
Conditional on the later groups, its probability is at most
$\rho_{g_j}$, while the later equations do not involve $G_j$.
Successively integrating over $G_1,\ldots,G_d$ gives
\begin{align*}
\begin{aligned}
q^{-u}\le q^{d-k}|U\cap E^k| =q^d\mathbb P(Bb=0) \le \prod_{j=1}^d q\rho_{g_j} \le \left(1+c(k-2d)_+\right)^{-1/2}.
\end{aligned}
\end{align*}
Here we used $\prod_j(1+x_j)\ge1+\sum_jx_j$ for $x_j\ge0$
and $\sum_j(g_j-2)_+\ge(k-2d)_+$.
Consequently,
\begin{align*}
d\ge \frac{k}{2}-\frac{q^{2u}-1}{2c}
\ge \frac{k}{2}-Cq^{2u}.
\end{align*}
This proves \eqref{eq:codim} after increasing
$C_{\ref{thm:count}}$ if necessary.
\end{proof}

\subsection{Approximate interpolation}

We next count approximate relations. The columns of an invertible
matrix $W\in E^{r\times r}$ form a basis of $\R^r$.
The linear map sending these basis vectors to the corresponding
columns of $Y\in E^{d\times r}$ is $x\mapsto YW^{-1}x$.
We count the matrices $Y$ for which this map sends many points
of $E^r$ close to $E^d$.

\begin{proposition}\label{prop:approxcount}
There are positive constants
$C'_{\ref{prop:approxcount}}$ depending only on $E$ and
$C_{\ref{prop:approxcount}}$ depending only on $q$
such that the following holds.
Let $r\ge1$ and $d\ge0$ be integers with $K=r+d\ge2$,
and set
\begin{align*}
\varepsilon
=\exp\{-C'_{\ref{prop:approxcount}}\log^2(K+1)\}.
\end{align*}
For every invertible $W\in E^{r\times r}$ and every $u\ge0$,
the number of matrices $Y\in E^{d\times r}$ satisfying
\begin{align*}
\left|\{x\in E^r:
\dist_\infty(YW^{-1}x,E^d)\le\varepsilon\}\right|
\ge q^{r-u}
\end{align*}
is at most
\begin{align*}
\exp\{C_{\ref{prop:approxcount}}
[K(u+1)+d\log(d+1)]\}.
\end{align*}
\end{proposition}

To obtain a bound uniform in $W$, we count the rows of $Y$
directly. We first extend \eqref{eq:projectioncount} to points
lying close to an affine subspace.

Let $L\subset\R^r$ be an affine subspace of dimension $\ell$.
If $0\le\rho\le\Delta/[4(r+1)]$, then
\begin{align}
\left|\{a\in E^r:\dist_\infty(a,L)\le\rho\}\right|
\le q^\ell.
\label{eq:tubecount}
\end{align}
For $\ell=r$, this is immediate.
Suppose $1\le\ell<r$.
Choose a basis matrix for the direction space of $L$ and
an $\ell\times\ell$ minor of maximal absolute determinant.
After permuting coordinates, write
\begin{align*}
L=\{(x,Ax+b):x\in\R^\ell\},
\end{align*}
using the selected coordinates as the first $\ell$ coordinates.
Cramer's rule and maximality of the chosen minor give
$|A_{ij}|\le1$.

If $(x,y)\in E^r$ lies within distance $\rho$ of $L$,
choose $(x',Ax'+b)\in L$ within that distance.
Since $|A_{ij}|\le1$,
\begin{align*}
\|y-Ax-b\|_\infty\le(\ell+1)\rho.
\end{align*}
Hence any two such points $(x,y)$ and $(x,\widetilde y)$
with the same first $\ell$ coordinates satisfy
\begin{align*}
\|y-\widetilde y\|_\infty
\le2(\ell+1)\rho<\Delta.
\end{align*}
The separation of the points of $E$ implies $y=\widetilde y$.
Projection onto the first $\ell$ coordinates is therefore
injective on the set being counted, proving \eqref{eq:tubecount}.
When $\ell=0$, the same separation argument gives at most
one point.

The next lemma replaces $p$ input coordinates by retained
input coordinates and $h$ of the $m$ output coordinates.
The resulting input dimension satisfies $p'\le C_qm$,
while the bound on the density loss improves from $v$ to
$v-2h$. This gain compensates for the cost of specifying
the corresponding rows of $Y$ in the recursive count.
The new inputs uniquely determine the discrete output $f(x)$.

\begin{lemma}\label{lem:compression}
There are constants $C_{\ref{lem:compression}}\ge1$
depending only on $q$ and
$C'_{\ref{lem:compression}},c_{\ref{lem:compression}}>0$
depending only on $E$ such that the following holds.
Let $m,p\ge1$ be integers, $B\in\R^{m\times p}$, $v\ge0$,
and $0<\varepsilon\le c_{\ref{lem:compression}}(m+p)^{-4}$.
Suppose that $\mathcal S\subseteq E^p$ has cardinality
at least $q^{p-v}$ and that $f:\mathcal S\to E^m$ satisfies
\begin{align*}
\norm{Bx-f(x)}_\infty\le\varepsilon,
\qquad x\in\mathcal S.
\end{align*}
Then there are sets $J\subseteq[p]$ and $I\subseteq[m]$,
with $h=|I|$ and $p'=|J|+h$, such that
\begin{align*}
p'\le\min(p,C_{\ref{lem:compression}}m),
\qquad 2h\le v.
\end{align*}
Write $B_I$ for the rows of $B$ indexed by $I$,
and $f_I$ for the corresponding coordinates of $f$.
Define $T:\R^p\to\R^{p'}$ and
$\psi:\mathcal S\to E^{p'}$ by
\begin{align*}
Tx=(x_J,B_Ix),
\qquad \psi(x)=(x_J,f_I(x)).
\end{align*}
The map $T$ has full row rank and admits a right inverse
$U:\R^{p'}\to\R^p$ determined by $J,I,B_I$.
Setting $B'=BU$, we have
\begin{align*}
\norm{B-B'T}_{2\to\infty}
\le C'_{\ref{lem:compression}}(m+p)^{5/2}\varepsilon.
\end{align*}
Moreover, $f(x)$ is uniquely determined by $\psi(x)$, and
\begin{align*}
|\psi(\mathcal S)| \ge q^{p'-(v-2h)},\qquad
\norm{B'\psi(x)-f(x)}_\infty \le C'_{\ref{lem:compression}}(m+p)^3\varepsilon,
\quad x\in\mathcal S.
\end{align*}
\end{lemma}

\begin{proof}
We construct $J$ and $I$ so that $B$ factors approximately
through $T$ and $\psi(x)$ determines $f(x)$.
The required gain of $2h$ in the density loss will follow
from a bound on the sizes of the fibers of $\psi$.
We begin with the small ball estimate used to obtain this bound.

\noindent\textit{A small ball estimate.}
Fix $L\in4\mathbb N$, sufficiently large in terms of $q$.
Suppose that $D\in\R^{h\times p}$ contains $L$ disjoint
column blocks $G_1,\ldots,G_L$, each with
$j=\lceil h/2\rceil$ columns and
$s_{\min}(G_\ell)\ge\kappa>0$.
We claim that, for $b$ uniform on $E^p$,
\begin{align*}
\sup_{y\in\R^h}
\PP\left(\norm{Db-y}_2\le\rho\sqrt h\right)
\le q^{-3h},
\qquad 0<\rho\le c\Delta\kappa,
\end{align*}
where $c>0$ is an absolute constant.

Condition on the coordinates outside the selected blocks.
Applying Proposition 2.7 in \cite{FJZ} to the independent
block sums gives
\begin{align*}
\sup_{y\in\R^h}
\PP\left(\norm{Db-y}_2\le\rho\sqrt h\right)
\le
2^h\max_{\ell\in[L]}\sup_{z\in\R^h}
\PP\left(\norm{G_\ell\zeta-z}_2\le4\rho\sqrt h\right),
\end{align*}
where $\zeta\in\R^j$ has independent coordinates, each
distributed as a sum of $L/2$ independent differences of
independent uniform random variables on $E$.
Lemma~\ref{lem:scalar} gives
$\Lc(\zeta_i,\Delta/4)\le C/\sqrt L$, with an absolute
constant $C$.
Since $s_{\min}(G_\ell)\ge\kappa$, each probability on
the right is bounded by the probability that $\zeta$
lies in a Euclidean ball of radius $4\rho\sqrt h/\kappa$.
For sufficiently small $c$, this radius is at most
$(\Delta/4)\sqrt j$.
Such a ball can be covered by at most $C^j$ cubes of
side length $\Delta/2$, so independence gives
\begin{align*}
\sup_{y\in\R^h}
\PP\left(\norm{Db-y}_2\le\rho\sqrt h\right)
\le 2^h(C/\sqrt L)^j.
\end{align*}
As $j\ge h/2$, choosing $L=L(q)$ sufficiently large
proves the claim.

\smallskip
\noindent\textit{Choosing the input coordinates.}
Set
$\delta_0=C_0\varepsilon\sqrt{mp}$, where
$C_0=C_0(E)$ will be chosen sufficiently large.
In the rest of the proof, constants $C$ may depend only
on $E$.

Start with $J=\emptyset$ and $\delta=\delta_0$.
At each stage, set $R=[p]\setminus J$, and write $B_R$
for the submatrix of columns indexed by $R$.
Let $h$ be the number of singular values of $B_R$
greater than $\delta$, and let $H$ be the corresponding
left singular space. Thus
\begin{align*}
\norm{(I-P_H)B_R}\le\delta.
\end{align*}
If $h=0$, take $H=\{0\}$ and stop.

Suppose that $h>0$, and set $j=\lceil h/2\rceil$.
We try to select $L$ disjoint blocks of $j$ columns,
temporarily removing each selected block.
Let $M$ be the matrix of columns still available.
Whenever $s_j(M)>2\delta$, we have
\begin{align*}
s_j(P_HM)
\ge s_j(M)-\norm{(I-P_H)M}
>\delta.
\end{align*}
Project $P_HM$ onto its top $j$ left singular space
and express the result in orthonormal coordinates as
$\widetilde M$.
Choose a $j\times j$ column submatrix $M_0$ of maximal
absolute determinant, and write $\widetilde M=M_0V$.
Maximality gives $|V_{i\ell}|\le1$, so
$\norm V\le\sqrt{jp}$ and
\begin{align*}
s_{\min}(M_0)\ge\frac{\delta}{\sqrt{jp}}.
\end{align*}
The corresponding column block of $P_HB_R$ satisfies
the same lower bound.

If $L$ blocks are found, leave them in $R$ and stop.
Otherwise, add the indices of the columns selected at this stage to $J$, replace $\delta$ by $2\delta$, and repeat.
At a failed stage, the remaining matrix has fewer than
$j$ singular values greater than $2\delta$.
The new value of $h$ therefore satisfies
\begin{align*}
h'\le\lceil h/2\rceil-1<h/2.
\end{align*}
Each failed stage adds fewer than $Lh$ indices to $J$.
Consequently, the procedure terminates with
\begin{align*}
|J|\le2Lm,
\qquad
\delta\le2(m+1)\delta_0.
\end{align*}
At termination, either $h=0$, or $P_HB_R$ contains
$L$ disjoint blocks of $j=\lceil h/2\rceil$ columns
with least singular values at least $\delta/\sqrt{jp}$.

\smallskip
\noindent\textit{Choosing the output rows.}
Suppose that $h>0$.
Choose $I\subseteq[m]$, with $|I|=h$, to maximize
$\det(AA^T)$, where $A=B_{I,R}$.
Since $\operatorname{rank}B_R\ge h$, the rows of $A$
are independent.
Write
\begin{align*}
B_R=\Gamma A+E_0,
\qquad E_0A^T=0.
\end{align*}
We have
\begin{align*}
|\Gamma_{\ell i}|\le1,
\qquad
\norm{(E_0)_{\ell\cdot}}_2
\le\sqrt{h+1}\,\delta.
\end{align*}

For the first bound, replacing the $i$-th selected row
by row $\ell$ gives a volume at least
$|\Gamma_{\ell i}|$ times the original volume, since
the residual is orthogonal to the selected rows.
Maximality therefore gives $|\Gamma_{\ell i}|\le1$.

For the second bound, append any row outside $I$ to $A$,
and let $Q$ be the Gram matrix of these $h+1$ rows.
Its smallest eigenvalue is at most $\delta^2$, and
each principal minor of order $h$ is at most
$\det(AA^T)$.
Hence
\begin{align*}
\det(AA^T)\norm{(E_0)_{\ell\cdot}}_2^2
=\det Q
\le(h+1)\delta^2\det(AA^T).
\end{align*}
The rows indexed by $I$ have zero residual, completing
the proof of both bounds.

\smallskip
\noindent\textit{The reduced system.}
Set $p'=|J|+h$ and
\begin{align*}
Tx=(x_J,B_Ix).
\end{align*}
For $h>0$, the independence of the rows of $A$ implies
that $T$ has full row rank.
Define
\begin{align*}
U(z,w)_J=z,\quad U(z,w)_R=A^T(AA^T)^{-1}(w-B_{I,J}z).
\end{align*}
Then $U$ is determined by $J,I,B_I$ and satisfies
$TU=I_{p'}$.
Since $E_0A^T=0$, we obtain
\begin{align*}
B'=BU=(B_J-\Gamma B_{I,J},\Gamma).
\end{align*}
Thus $B-B'T$ is zero on $J$ and equals $E_0$ on $R$.

If $h=0$, take $I=\emptyset$ and let $U$ be coordinate
inclusion from $\R^J$ into $\R^p$.
Then $T$ is coordinate projection, $B'=B_J$, and
the residual has norm at most $\delta$.
This also covers $p'=0$.
In either case,
\begin{align*}
\norm{B-B'T}_{2\to\infty}
\le\sqrt{h+1}\,\delta
\le C(m+p)^{5/2}\varepsilon.
\end{align*}
Full row rank gives $p'\le p$, while
$|J|\le2Lm$ and $h\le m$ give
$p'\le(2L+1)m$.
We may therefore take
$C_{\ref{lem:compression}}=2L+1$.

Define $\psi(x)=(x_J,f_I(x))$.
The coefficients of $B'$ multiplying the last $h$
coordinates have absolute value at most one.
Using $\norm x_2\le C\sqrt p$ for $x\in E^p$, we obtain
\begin{align*}
\norm{B'\psi(x)-f(x)}_\infty
&\le
\norm{B'(\psi(x)-Tx)}_\infty
+\norm{(B'T-B)x}_\infty
+\norm{Bx-f(x)}_\infty\\
&\le(h+1)\varepsilon
+C\sqrt p\sqrt{h+1}\,\delta\\
&\le C(m+p)^3\varepsilon.
\end{align*}
Choose $c_{\ref{lem:compression}}$ sufficiently small
that the last bound is less than $\Delta/4$.
Since distinct elements of $E$ are separated by at
least $\Delta$, two points with the same value of
$\psi$ have the same value of $f$.

It remains to bound the fibers of $\psi$.
Fix a nonempty fiber $\psi(x)=(z,w)$, and denote the
common value of $f(x)$ on this fiber by $f_0$.
Every point in the fiber satisfies $x_J=z$ and
\begin{align*}
\norm{P_H(B_Rx_R-(f_0-B_Jz))}_2
\le\varepsilon\sqrt m.
\end{align*}
For $h>0$, identify $H$ with $\R^h$ and apply the small
ball estimate to $P_HB_R$, with
\begin{align*}
\rho=\varepsilon\sqrt{m/h},
\qquad
\kappa=\delta/\sqrt{jp}.
\end{align*}
Indeed,
\begin{align*}
\frac{\rho}{\kappa}
=\frac{\varepsilon\sqrt{mp}}{\delta}\sqrt{\frac jh}
\le\frac{\varepsilon\sqrt{mp}}{\delta_0},
\end{align*}
so $\rho\le c\Delta\kappa$ when $C_0(E)$ is sufficiently
large.
The fiber therefore contains at most
$q^{p-|J|-3h}$ points.
For $h=0$, the same bound is simply the size of a
coordinate fiber.
Consequently,
\begin{align*}
|\psi(\mathcal S)|
\ge |\mathcal S|q^{|J|+3h-p}
\ge q^{|J|+3h-v}
=q^{p'-(v-2h)}.
\end{align*}
Finally, $\psi(\mathcal S)\subseteq E^{p'}$, so
$|\psi(\mathcal S)|\le q^{p'}$ and hence $2h\le v$.
\end{proof}

\begin{proof}[Proof of Proposition~\ref{prop:approxcount}]
The case $d=0$ is immediate, so assume $d\ge1$.
Write $L=C_{\ref{lem:compression}}$, and let
$\varepsilon_0$ denote the error in the statement.
We count the original rows of $Y$ throughout.
At input dimension $p$, each such row will have at most
$q^p$ possibilities; the decrease in density loss supplied
by Lemma~\ref{lem:compression} will absorb the cost of
specifying the selected rows.

\smallskip
\noindent\textit{Reduction.}
At each stage, we consider original rows
$a_1,\ldots,a_m\in E^r$ and fixed matrices
$F\in\R^{p\times r}$, $R\in\R^{r\times p}$ with $FR=I_p$.
Let $B$ have rows $b_i=a_iR$. Right multiplication by $R$ gives the current coefficient
rows, while right multiplication by $F$ reconstructs the
original rows up to error $\rho$ in the $\ell_\infty$ norm.
We retain the conditions
\begin{align*}
\norm{a_i-b_iF}_\infty \le\rho,
\quad i\in[m], \qquad
\left|\{x\in E^p:\dist_\infty(Bx,E^m)\le\varepsilon\}\right|
&\ge q^{p-v}.
\end{align*}
Initially, $m=d$, $p=r$, $F=W$, $R=W^{-1}$,
$v=u$, $\rho=0$ and $\varepsilon=\varepsilon_0$.

Provided $\rho\le\Delta/[4(r+1)]$, every original row
lies in the $\rho$-neighborhood of the row space of $F$.
This space has dimension $p$, so \eqref{eq:tubecount}
gives at most $q^p$ choices for each row.

For $p\ge1$, apply Lemma~\ref{lem:compression} to $B$,
choosing an approximating point in $E^m$ for each
approximate solution.
Fix the resulting sets $J,I$, with $h=|I|$, and specify
the selected original rows.
For fixed $J,I$, there are at most $q^{ph}$ choices
for these rows.
They determine $B_I$, hence $T,U$ and the new maps
\begin{align*}
F'=TF,\qquad R'=RU,\qquad F'R'=I_{p'}.
\end{align*}
The new coefficient rows are $a_iR'=b_iU$.
The lemma gives input dimension
$p'\le\min(p,Lm)$ and density loss at most $v-2h$,
where $2h\le v$.

We apply this reduction once to the whole matrix if
$r>Ld$.
Thereafter, we split each row array into two fixed
blocks of sizes $\lfloor m/2\rfloor$ and
$\lceil m/2\rceil$, and reduce each block separately.
We stop when $m=1$ or $p=0$.
Thus every system after the possible initial reduction
satisfies $p\le Lm$, and each branch contains
$O(\log(K+1))$ reductions.

We verify that the errors remain small throughout.
Each row of $F'$ is either an old row of $F$ or a row
$a_iRF$ corresponding to a selected original row.
As long as $\rho\le1$, the entries of all these matrices
are bounded in terms of $E$, and hence $\norm F\le CK$.
Moreover,
\begin{align*}
a_i-a_iR'F'
=a_i-a_iRF+(b_i-b_iUT)F.
\end{align*}
The compression lemma therefore allows the error
parameters to be updated by
$\varepsilon'\le A\varepsilon$ and
$\rho'\le\rho+A\varepsilon$, where
$A=C(K+1)^5\ge2$ and $C$ depends only on $E$.
These bounds involve $\norm F$ but not $\norm R$,
which gives the required uniformity in $W$.
Choose $C'_{\ref{prop:approxcount}}$ sufficiently large
that, along every branch,
\begin{align*}
\max\{\varepsilon,\rho\}
\le e^{C\log^2(K+1)}\varepsilon_0
\le
\min\left\{
c_{\ref{lem:compression}}K^{-4},
\frac{\Delta}{4(K+1)},1
\right\}.
\end{align*}
Induction from the initial system now justifies every
application of the compression lemma and the tube estimate.

\smallskip
\noindent\textit{Counting.}
We prove that, with all preceding choices fixed, the
number of possible ordered arrays of original rows
at any stage with $p\le Lm$ is at most
\begin{align}
\exp\{C[m\log(m+1)+m(v+1)]\},
\label{eq:recursivecount}
\end{align}
where $C$ depends only on $q$.
The bound is uniform over the maps arising in the reduction.

We argue by induction on $m$.
For $p=0$, the tube estimate gives at most one row
at each position.
For $m=1$, there are at most $q^p\le q^L$ choices.
Now suppose that $m\ge2$, and set
$m_1=\lfloor m/2\rfloor$ and $m_2=\lceil m/2\rceil$.
The approximate solutions for the whole system satisfy
the constraints of each block, so both blocks start
with density loss at most $v$.

For block $i$, choose $J_i,I_i$ and specify its
$h_i=|I_i|$ selected original rows.
There are at most $2^{2p+m}$ choices for the four sets.
For fixed sets, the selected rows have at most
$q^{p(h_1+h_2)}$ possibilities and determine the
reduced maps.
Applying the inductive bound to the two reduced systems,
the logarithm of the count is at most
\begin{align*}
(2p+m)\log2
+Cm(v+1)
+C\sum_{i=1}^2m_i\log(m_i+1)
+\max_{h_1,h_2}\sum_{i=1}^2
       (p\log q-2Cm_i)h_i,
\end{align*}
where $0\le2h_i\le v$.

Since $p\le Lm\le3Lm_i$, the last term is nonpositive
when $C$ is sufficiently large in terms of $q$.
The remaining cost $(2p+m)\log2=O_q(m)$ is absorbed
by the inequality
\begin{align*}
m\log(m+1)-\sum_{i=1}^2m_i\log(m_i+1)
\ge m\log(4/3),
\end{align*}
after increasing $C$ if necessary.
This proves \eqref{eq:recursivecount}.
The two ordered blocks determine the original row array.

If $r\le Ld$, apply \eqref{eq:recursivecount} directly
with $m=d$ and $v=u$.
If $r>Ld$, the initial reduction requires choosing
$J,I$ and $h\le u/2$ original rows.
The total number of choices is at most
$2^{r+d}q^{ru/2}=\exp\{O_q(K(u+1))\}$.
The resulting system has $d$ rows, input dimension
at most $Ld$, and density loss at most $u-2h$,
so \eqref{eq:recursivecount} applies.
In either case, the number of original matrices $Y$
is at most
\begin{align*}
\exp\{C[K(u+1)+d\log(d+1)]\},
\end{align*}
as required.
\end{proof}

\subsection{Bases in dense subsets}

We conclude with two lemmas used to choose bases in dense
subsets of $E^r$. The first is a small ball estimate for
orthogonal projections, uniform over the shift, which will
also be used in the Fourier integral estimate in
Section~\ref{sec:averaging}. The second gives many bases
with controlled interpolation coefficients and determinants.

\begin{lemma}\label{lem:projection}
Let $b$ be uniform on $E^s$, and let
$P:\R^s\to\R^s$ be an orthogonal projection of rank $h\ge1$.
There are constants
$c_{\ref{lem:projection}},c'_{\ref{lem:projection}}>0$
depending only on $E$ such that, for every $v\in\R^s$,
\begin{align*}
\PP\left(
\norm{P(b-v)}_2^2\le c_{\ref{lem:projection}}h
\right)
\le2e^{-c'_{\ref{lem:projection}}h}.
\end{align*}
\end{lemma}

\begin{proof}
Let $b'$ be an independent copy of $b$, and let $\sigma^2>0$
be the variance of a random variable uniform on $E$.
The coordinates of $b-b'$ are independent, centered and
bounded, with variance $2\sigma^2$.
Moreover,
\begin{align*}
\EE\norm{P(b-b')}_2^2=2\sigma^2h,
\qquad
\norm P=1,
\qquad
\norm P_{\mathrm{HS}}^2=h.
\end{align*}
If both $\norm{P(b-v)}_2$ and $\norm{P(b'-v)}_2$ are at
most $\sigma\sqrt h/2$, then
$\norm{P(b-b')}_2\le\sigma\sqrt h$.
Thus independence and the Hanson--Wright inequality
\cite{HW} give
\begin{align*}
\PP\left(\norm{P(b-v)}_2^2\le\sigma^2h/4\right)^2 \le
\PP\left(\norm{P(b-b')}_2^2\le\sigma^2h\right) \le2e^{-ch},
\end{align*}
where $c>0$ depends only on $E$.
Taking square roots proves the assertion with
$c_{\ref{lem:projection}}=\sigma^2/4$ and
$c'_{\ref{lem:projection}}=c/2$.
\end{proof}

The next lemma provides many bases with controlled
interpolation coefficients.
We will average the corresponding slab indicators before
integration in Section~\ref{sec:averaging}.
When a basis with least singular value at least $\zeta$
is available, the determinant estimate involves only
$g=O_E(u+1)$ factors of $\zeta$, where $u$ is the density loss.

\begin{lemma}\label{lem:densebases}
There is a constant $C_{\ref{lem:densebases}}\ge1$
depending only on $E$ such that the following holds.
Let $r\ge1$, $u\ge0$ and $a\in E\setminus\{0\}$.
Suppose that $\mathcal S\subset E^r$ spans $\R^r$,
contains $a\one_r$, and satisfies $|\mathcal S|\ge q^{r-u}$.
Set
\begin{align*}
m=(r-1-\lceil C_{\ref{lem:densebases}}(u+1)\rceil)_+, \quad g=r-1-m\le2C_{\ref{lem:densebases}}(u+1).
\end{align*}
There are at least
$e^{-C_{\ref{lem:densebases}}r}|\mathcal S|^m$
ordered bases
$W=(a\one_r,x_1,\ldots,x_{r-1})$
with columns in $\mathcal S$ satisfying
\begin{align}
\sup_{x\in\mathcal S}\norm{W^{-1}x}_1
\le C_{\ref{lem:densebases}}(r+1)^4.
\label{eq:barycentric}
\end{align}
If $\mathcal S$ contains a basis $Q$ with
$s_{\min}(Q)\ge\zeta$, where $0<\zeta\le1$,
these bases can also be chosen to satisfy
\begin{align}
|\det W|
\ge e^{-C_{\ref{lem:densebases}}r\log(r+1)}\zeta^g.
\label{eq:basisdet}
\end{align}
\end{lemma}

\begin{proof}
We first choose many partial bases containing $a\one_r$
whose least singular values are at least $c/r$.
We then complete each partial basis by maximizing the
absolute value of the determinant while keeping the chosen
columns fixed.
Throughout, $c,C>0$ depend only on $E$.

\smallskip
\noindent\textit{A spectral lower bound.}
Let $P$ be the orthogonal projection onto $\one_r^\perp$,
let $X$ be uniform on $\mathcal S$, and define the
second moment matrix
\begin{align*}
\Sigma=\EE[(PX)(PX)^T],
\end{align*}
viewed as an operator on $\one_r^\perp$.
Let $c_0,c_1>0$ be the constants in
Lemma~\ref{lem:projection}.
If $h\ge1$ eigenvalues of $\Sigma$ are less than $c_0/4$,
let $P_0$ project onto the corresponding eigenspace.
Markov's inequality gives
$\PP(\norm{P_0X}_2^2\le c_0h/2)\ge1/2$.
For $b$ uniform on $E^r$, the density assumption and
Lemma~\ref{lem:projection} therefore give
\begin{align*}
q^{-u}/2
\le \PP\left(\norm{P_0b}_2^2\le c_0h/2\right)
\le2e^{-c_1h}.
\end{align*}
Thus $h\le C(u+1)$.
Choosing $C_{\ref{lem:densebases}}$ sufficiently large
ensures that $\Sigma$ has at least $m$ eigenvalues
bounded below by a positive constant depending only on $E$.

\smallskip
\noindent\textit{Many partial bases.}
Assume first that $m\ge1$.
Take independent copies $X_1,\ldots,X_m$ of $X$,
and set
\begin{align*}
Z=(PX_1,\ldots,PX_m),
\qquad G=Z^TZ.
\end{align*}
Cauchy--Binet and independence give
\begin{align*}
\EE\det G=m!e_m(\Sigma),\quad
\EE\operatorname{tr}(\operatorname{adj}G) =m!e_{m-1}(\Sigma),
\end{align*}
where $e_j$ is the $j$th elementary symmetric function
of the eigenvalues and $e_0=1$.
The first identity follows by expanding the squared
minors of $Z$; the second follows by deleting one column
and summing the resulting $m$ Gram determinants.

Weight the law of the tuple by $\det G/\EE\det G$,
and write $\PP_*,\EE_*$ for probability and expectation
under this law.
It is supported on invertible $G$, and
\begin{align*}
\EE_*\operatorname{tr}G^{-1}
\le
\frac{\EE\operatorname{tr}(\operatorname{adj}G)}
     {\EE\det G}
=
\frac{e_{m-1}(\Sigma)}{e_m(\Sigma)}
\le Cm.
\end{align*}
The first inequality uses the nonnegative contribution
of singular Gram matrices to the adjugate trace.
For the last inequality, fix $m$ eigenvalues bounded
below by $c$.
Every product of $m-1$ eigenvalues can be extended using
an unused eigenvalue from this set, and each resulting
product is counted at most $m$ times.
Hence $ce_{m-1}(\Sigma)\le me_m(\Sigma)$.

Markov's inequality now gives
\begin{align*}
\PP_*(\mathcal G)\ge\frac12,
\qquad
\mathcal G=\{s_{\min}(Z)\ge c/\sqrt m\}.
\end{align*}
Since $\det G\le(Cr)^m$ and
$\EE\det G\ge m!c^m$, returning to the original law yields
\begin{align*}
\PP(\mathcal G) \ge
\frac{\EE[\det G\,\one_{\mathcal G}]}{(Cr)^m} \ge\frac{m!c^m}{2(Cr)^m}
\ge e^{-Cr}.
\end{align*}
The last inequality follows from $m!\ge(m/e)^m$ and
$m\log(r/m)\le r/e$.
There are therefore at least
$e^{-Cr}|\mathcal S|^m$ ordered tuples
$(x_1,\ldots,x_m)$ with this property.

For each such tuple, set
$A=(a\one_r,x_1,\ldots,x_m)$.
If $y=A\alpha$, projection onto $\one_r^\perp$
and then onto $\R\one_r$ gives, respectively,
\begin{align*}
\norm{(\alpha_1,\ldots,\alpha_m)}_2 \le C\sqrt m\norm y_2,\quad |\alpha_0|\le C(1+m)\norm y_2.
\end{align*}
It follows that $s_{\min}(A)\ge c/r$.
When $m=0$, take $A=(a\one_r)$; the same bound holds,
and there is one such partial basis.

\smallskip
\noindent\textit{Completion.}
For each partial basis $A$, choose $g$ further columns
from $\mathcal S$ to maximize $|\det W|$, keeping $A$
fixed and resolving ties in a fixed way.
Spanning ensures that the maximum is nonzero.
Write $W=(A,D)$ and, for $x\in\mathcal S$, write
$W^{-1}x=(\alpha,\beta)$.
Replacing a column of $D$ by $x$ and applying Cramer's
rule gives $\norm\beta_\infty\le1$.
Since $A\alpha=x-D\beta$ and $s_{\min}(A)\ge c/r$,
\begin{align*}
\norm\alpha_2
\le Cr\norm{x-D\beta}_2
\le Cr^{3/2}(1+g).
\end{align*}
Consequently,
\begin{align*}
\norm{W^{-1}x}_1
\le\sqrt{m+1}\norm\alpha_2+g
\le C(r+1)^4,
\end{align*}
which proves \eqref{eq:barycentric}.
Distinct ordered tuples give distinct ordered completed
bases, yielding the required multiplicity.

Finally, suppose that the additional basis $Q$ is available.
We compare the maximal completion with a completion
using columns of $Q$.
At each incomplete stage, choose a unit vector $v$
orthogonal to the current span.
Since $\norm{Q^Tv}_2\ge\zeta$, some column of $Q$
has distance at least $\zeta/\sqrt r$ from that span.
Choosing such a column at each step gives a completion
whose determinant has absolute value at least
\begin{align*}
\vol(A)(\zeta/\sqrt r)^g
\ge(c/r)^{m+1}(\zeta/\sqrt r)^g.
\end{align*}
The same maximal completion $W$ satisfies this lower bound.
Since $m+1+g=r$, it implies \eqref{eq:basisdet}
after choosing $C_{\ref{lem:densebases}}$ sufficiently large.
\end{proof}

\section{Fourier estimates for concentration functions}\label{sec:averaging}

We prove the averaging estimate used in the inversion of randomness
argument. Rescaling and rounding an incompressible normal at scale $T$
produces an integer vector in a product set at scale
$N=\Theta(T^{-1})$. Theorem~\ref{thm:averaging} shows that, throughout
the range $N\le q^n(n+1)^{-D}$, only an $e^{-\omega(n)}$ proportion
of these vectors can have concentration at least
$L_{\ref{thm:averaging}}/N$. The constant in the exponent can be
chosen arbitrarily large without changing this threshold.

We next introduce the admissible product sets from Section~4 in \cite{Tik}.
\begin{definition}\label{def:admissible}
For positive integers $N,n$, $K\ge1$, and $\delta\in(0,1]$, a
product $\mathcal A=\prod_{i=1}^n A_i\subset\Z^n$ is
admissible with parameters $(N,n,K,\delta)$ if:
\begin{enumerate}[label=(\roman*),leftmargin=2em]
\item $A_i=-A_i$;
\item for $i>\delta n$, $A_i$ is an integer interval of size at
least $2N+1$;
\item for $i\le\delta n$, $A_i$ is the union of two integer
intervals, each of size at least $N$, and
$A_i\cap[-N,N]=\emptyset$;
\item $\prod_i|A_i|\le(KN)^n$ and $\max_{a\in A_i}|a|<nN$ for every $i$.
\end{enumerate}
\end{definition}

\begin{theorem}\label{thm:averaging}
Let $E\subset\R$ be a fixed set of $q\ge2$ points.
For fixed $\delta\in(0,1]$ and $K\ge1$, there is
$L_{\ref{thm:averaging}}=L_{\ref{thm:averaging}}(E,\delta,K)>0$
independent of $M,D$ such that for every $M\ge1$ and $D>1$,
the following holds for all
$n\ge n_{\ref{thm:averaging}}(E,\delta,K,M,D)$ and all integers
\begin{align*}
1\le N\le q^n(n+1)^{-D}:
\end{align*}
a uniform vector $X$ on any set admissible with parameters $(N,n,K,\delta)$ satisfies
\begin{align}
\PP_X\left(\Lc_b\left(\sum_{i=1}^n b_iX_i,\sqrt n\right)
                         \ge\frac{L_{\ref{thm:averaging}}}{N}\right)\le e^{-Mn}.
 \label{eq:averaging}
\end{align}
The variables $b_i$ are independent and uniformly distributed on $E$, and are independent of $X$;
$\Lc_b$ denotes concentration with respect to the variables $b_i$,
with the coefficients fixed.
\end{theorem}

Equivalently, at most $e^{-Mn}|\mathcal A|$ integer vectors in
$\mathcal A$ satisfy the concentration condition in \eqref{eq:averaging}.
Related averaging estimates appear in Corollary~4.3 in \cite{Tik} and in the work of Jain, Sah and Sawhney \cite{JSS}.
For uniform distributions on finite sets, the latter estimate applies in a range separated from $q^n$ by a fixed exponential factor. Here the upper endpoint is only a polynomial factor
below $q^n$.

For fixed coefficients, replacing $E$ by $E-a$, where $a\in E$, only translates the linear form and leaves its concentration function unchanged. We therefore assume $0\in E$ throughout the proof of the
averaging theorem. This does not translate the random matrix in
Theorem~\ref{thm:main}. Write $\Delta=\min\{|a-b|:a,b\in E,\ a\ne b\}>0$. 

We pass to continuous coefficients to turn the count of integer
vectors into volume and integral estimates suited to Fourier analysis.
Adding independent bounded perturbations preserves concentration at
the $\sqrt n$ scale up to constant factors, with probability bounded below.
The construction and comparison are given in the proof below.

For the continuous estimate, let $Z_1,\ldots,Z_n$ be independent.
For $i>\lfloor\delta n\rfloor$, the variable $Z_i$ is uniform on
a bounded interval $I_i$; for $i\le\lfloor\delta n\rfloor$, it is
uniform on the union $I_i$ of two disjoint bounded intervals.
Every component interval has length at least $N$. We assume
\begin{align*}
I_i\subset[-2nN,2nN],\qquad
 \prod_{i=1}^n|I_i|\le(KN)^n,\qquad
 I_i\cap(-N/2,N/2)=\emptyset\quad(i\le\lfloor\delta n\rfloor),
\end{align*}
where $|I_i|$ denotes Lebesgue measure. Write $\mu_i$ for the
density of $Z_i$.

Let $B=\max_{a\in E}|a|$ and $\beta=1+B$, and define
\begin{align*}
f_n(x)=(2\pi\beta^2n)^{-1/2}\exp\{-x^2/(2\beta^2n)\},
 \qquad F_z(x)=\EE_b f_n(x+\langle b,z\rangle).
\end{align*}
The expectation is over independent variables
$b_1,\ldots,b_n$ uniform on $E$, with $z\in\R^n$ fixed.
For fixed $z$, $F_z$ is the density of $G-\langle b,z\rangle$,
where $G$ is an independent centered Gaussian of variance $\beta^2n$.
The next proposition therefore estimates a density obtained by Gaussian
smoothing at scale $\beta\sqrt n$; the comparison below turns its
supremum norm bound into the required concentration estimate.

\begin{proposition}\label{prop:smoothing}
For fixed $E,\delta,K$, there is a constant
$C_{\ref{prop:smoothing}}>0$ independent of $M,D$ such that
the following holds. For every $M\ge1$, $D>1$, and
$n\ge n_{\ref{prop:smoothing}}(E,\delta,K,M,D)$,
uniformly over integers $1\le N\le q^n(n+1)^{-D}$
and all random vectors $Z=(Z_1,\ldots,Z_n)$ satisfying the above assumptions,
\begin{align*}
\PP_Z\left(\norm{F_Z}_\infty>
       \frac{C_{\ref{prop:smoothing}}}{N\sqrt n}\right)\le e^{-Mn}.
\end{align*}
\end{proposition}

We first deduce the discrete averaging theorem from this estimate.

\begin{proof}[Proof of Theorem~\ref{thm:averaging} from
Proposition~\ref{prop:smoothing}]
Let $X$ be uniform on the admissible product set, let $U_i$ be
independent uniform variables on $[-1/2,1/2]$, independent of $X$,
and set $Z_i=X_i+U_i$. The variables $Z_i$ are independent and
uniform on $I_i=A_i+[-1/2,1/2]$, up to endpoints of measure zero.
Each integer interval becomes a real interval of the same measure
as its cardinality, so all the continuous support conditions hold.
For $i\le\lfloor\delta n\rfloor$, admissibility even gives
$|Z_i|\ge N+1/2$ almost surely.

Fix an integer vector $x$, and choose an interval $I_{x}$ centered at $t$
with radius $\sqrt n$ such that
\begin{align*}
\rho_x:=\Lc_b(\langle b,x\rangle,\sqrt n)
       =\PP_b(\langle b,x\rangle\in I_{x}).
\end{align*}
For each fixed $b\in E^n$, Chebyshev's inequality gives
\begin{align*}
\PP_U\bigl(|\langle b,U\rangle|>B\sqrt n\bigr)
 \le\frac1{12}.
\end{align*}
Define the lost probability mass by
\begin{align*}
D_x(U)
 :=\PP_b\bigl(\langle b,x\rangle\in I_x,\,
              |\langle b,U\rangle|>B\sqrt n\bigr).
\end{align*}
Interchanging expectations and averaging this estimate over $b$, we obtain
\begin{align*}
\EE_U D_x(U)=\EE_b\left[
   \one_{\{\langle b,x\rangle\in I_x\}}
   \PP_U\bigl(|\langle b,U\rangle|>B\sqrt n\bigr)
   \right]\le\frac{\rho_x}{8}.
\end{align*}
Markov's inequality therefore yields $\PP_U\bigl(D_x(U)>\rho_x/2\bigr)
 \le 1/4$. Whenever $\langle b,x\rangle\in I_x$ and
$|\langle b,U\rangle|\le B\sqrt n$, we have
$|\langle b,x+U\rangle-t|\le\beta\sqrt n$.
Since $f_n(y)\ge(C\sqrt n)^{-1}$ for
$|y|\le\beta\sqrt n$, where
$C=\sqrt{2\pi}\,\beta e^{1/2}$, it follows that
\begin{align*}
\norm{F_{x+U}}_\infty
 \ge F_{x+U}(-t)
 \ge\frac{\rho_x-D_x(U)}{C\sqrt n}.
\end{align*}
Averaging over $X$, we therefore obtain, for every $\rho>0$,
\begin{align*}
\PP_X\left(\Lc_b(\langle b,X\rangle,\sqrt n)\ge\rho\right)
 \le\frac43\PP_Z\left(
       \norm{F_Z}_\infty\ge\frac{\rho}{2C\sqrt n}\right).
\end{align*}
Take $L_{\ref{thm:averaging}}=4CC_{\ref{prop:smoothing}}$ as the threshold in the averaging theorem.
The event in the last probability, with
$\rho=L_{\ref{thm:averaging}}/N$, implies
$\norm{F_Z}_\infty\ge2C_{\ref{prop:smoothing}}/(N\sqrt n)>C_{\ref{prop:smoothing}}/(N\sqrt n)$.
Choose
$n_{\ref{thm:averaging}}(E,\delta,K,M,D)
\ge n_{\ref{prop:smoothing}}(E,\delta,K,M+1,D)$
and apply Proposition~\ref{prop:smoothing} with exponent $M+1$.
The resulting bound is $(4/3)e^{-(M+1)n}\le e^{-Mn}$.
The choice $L_{\ref{thm:averaging}}=4CC_{\ref{prop:smoothing}}$ depends only on $E,\delta,K$.
\end{proof}

\subsection{Decomposition of the Fourier integral}\label{sec:decomposition}

We use the Fourier transform convention
$\widehat f(t)=\int_\R e^{-2\pi itx}f(x)\,dx$.
Set
\begin{align*}
\omega_n(\theta)=\widehat f_n(\theta)
   =e^{-2\pi^2\beta^2n\theta^2},\qquad
 \varphi(t)=q^{-1}\sum_{a\in E}e^{2\pi iat},\qquad
 \Phi_z(\theta)=\prod_{i=1}^n\varphi(z_i\theta).
\end{align*}
We split the Fourier integral at the frequency scales
$N^{-1}$ and $n^2/N$. The factor $n^2$ is chosen for convenience.
Fourier inversion gives
\begin{align*}
F_z(x)
 &=q^{-n}\sum_{b\in E^n}f_n(x+\langle b,z\rangle) =\int_\R e^{2\pi ix\theta}\omega_n(\theta)
          \prod_{i=1}^n\left(q^{-1}\sum_{a\in E}e^{2\pi iaz_i\theta}\right)
          \,d\theta\\
&=P_z(x)+Q_z(x)+R_z(x),
\end{align*}
where
\begin{align*}
P_z(x)&=\int_{|\theta|\le N^{-1}}
       e^{2\pi ix\theta}\omega_n(\theta)\Phi_z(\theta)\,d\theta,\\
Q_z(x)&=\int_{N^{-1}<|\theta|\le n^{2}/N}
       e^{2\pi ix\theta}\omega_n(\theta)\Phi_z(\theta)\,d\theta,\\
R_z(x)&=\int_{|\theta|>n^{2}/N}
       e^{2\pi ix\theta}\omega_n(\theta)\Phi_z(\theta)\,d\theta.
\end{align*}
All three integrals converge absolutely. We estimate the three
terms in separate propositions.

\begin{proposition}\label{prop:low}
There is a constant $C_{\ref{prop:low}}>0$ depending only on
$E,\delta$ such that for
$n\ge n_{\ref{prop:low}}(E,\delta)$ and every
$z$ in the support of $Z$,
\begin{align*}
\norm{P_z}_\infty\le\frac{C_{\ref{prop:low}}}{N\sqrt n}.
\end{align*}
\end{proposition}

\begin{proof}
Fix distinct $a,a'\in E$ and set $h=a-a'\ne0$. If $b,b'$ are
independent and uniform on $E$, then
\begin{align*}
|\varphi(t)|^2
 =\EE e^{2\pi i(b-b')t}
 \le 1-\frac{2}{q^2}\left(1-\cos(2\pi ht)\right)
 \le e^{-c\norm{ht}_{\T}^2}.
\end{align*}
After changing the constant, the same bound holds for
$|\varphi(t)|$. Splitting an interval $I$ into periods of
$t\mapsto\norm{ht}_{\T}$ and using
$\int_{-1/2}^{1/2}e^{-cv u^2}\,du\le Cv^{-1/2}$ gives, for
$v\ge2$,
\begin{align}
\int_I|\varphi(t)|^v\,dt
 \le\frac{C}{\sqrt v}\left(|I|+|h|^{-1}\right).
 \label{eq:phiintegral}
\end{align}
Set $m=\lfloor\delta n\rfloor$. Since $|\varphi|\le1$, H\"older's
inequality gives
\begin{align*}
\int_{|\theta|\le1/N}|\Phi_z(\theta)|\,d\theta
 \le\prod_{i=1}^m
 \left(\int_{|\theta|\le1/N}
        |\varphi(z_i\theta)|^m\,d\theta\right)^{1/m}.
\end{align*}
For $i\le m$, the support assumptions give $|z_i|\ge N/2$.
Changing variables $t=z_i\theta$ and using
\eqref{eq:phiintegral},
\begin{align*}
\int_{|\theta|\le1/N}|\varphi(z_i\theta)|^m\,d\theta
 \le\frac{C}{N\sqrt m}.
\end{align*}
It follows that
\begin{align*}
\int_{|\theta|\le1/N}|\Phi_z(\theta)|\,d\theta
 \le\frac{C_{\ref{prop:low}}}{N\sqrt n}.
\end{align*}
Since $0<\omega_n\le1$, the same bound holds for
$\norm{P_z}_\infty$.
\end{proof}

\begin{proposition}\label{prop:intermediate}
For every $M\ge1$ and $n\ge n_{\ref{prop:intermediate}}(E,M)$,
\begin{align*}
\PP_Z\left(\norm{Q_Z}_\infty>\frac1{N\sqrt n}\right)\le e^{-Mn}.
\end{align*}
\end{proposition}

\begin{proof}
Let $I_N=\{\theta:1/N\le|\theta|\le n^{2}/N\}$ and
\begin{align*}
J(Z)=\int_{I_N}|\Phi_Z(\theta)|\,d\theta.
\end{align*}
For every fixed integer $v\ge2$, H\"older's inequality yields
\begin{align*}
J(Z)^v\le |I_N|^{v-1}
 \int_{I_N}\prod_{i=1}^n|\varphi(Z_i\theta)|^v\,d\theta.
\end{align*}
For $\theta\in I_N$, applying \eqref{eq:phiintegral} on each interval
supporting $Z_i$ gives
\begin{align*}
\EE|\varphi(Z_i\theta)|^v\le C/\sqrt v.
\end{align*}
The $Z_i$ are independent, and hence
\begin{align*}
\EE J(Z)^v
 \le |I_N|^v\left(\frac{C}{\sqrt v}\right)^n
 \le\left(\frac{2n^{2}}{N}\right)^v
      \left(\frac{C}{\sqrt v}\right)^n.
\end{align*}
Markov's inequality now gives
\begin{align*}
\PP\left(J(Z)>\frac1{N\sqrt n}\right) \le (N\sqrt n)^v\EE J(Z)^v \le(2n^{5/2})^v \left(\frac{C}{\sqrt v}\right)^n.
\end{align*}
Given $M$, choose a fixed $v$ with
$\log(C/\sqrt v)<-M-1$. Since
$v\log(2n^{5/2})=o(n)$, choose
$n_{\ref{prop:intermediate}}(E,M)$ so that
$v\log(2n^{5/2})\le n$ above this threshold. The last bound
is then at most $e^{-Mn}$. Finally,
$\norm{Q_Z}_\infty\le J(Z)$.
\end{proof}

\begin{proposition}\label{prop:far}
For every $M\ge1$ and $D>1$, for
$n\ge n_{\ref{prop:far}}(E,M,D)$ and all integers
$1\le N\le q^n(n+1)^{-D}$,
\begin{align*}
\PP_Z\left(\norm{R_Z}_\infty>\frac1{N\sqrt n}\right)\le e^{-Mn}.
\end{align*}
\end{proposition}

We prove Proposition~\ref{prop:far} in Section~\ref{sec:farproof}. Assuming it for now completes
the continuous averaging estimate.

\begin{proof}[Proof of Proposition~\ref{prop:smoothing} assuming
Proposition~\ref{prop:far}]
Set $C_{\ref{prop:smoothing}}=C_{\ref{prop:low}}+2$.
This choice depends only on $E,\delta$ and is made before $M,D$.
For prescribed $M,D$, apply Propositions~\ref{prop:far} and~\ref{prop:intermediate} with
exponent $M+1$.
Choose $n_{\ref{prop:smoothing}}(E,\delta,K,M,D)$ at least as large as
$n_{\ref{prop:low}}(E,\delta)$,
$n_{\ref{prop:far}}(E,M+1,D)$, and
$n_{\ref{prop:intermediate}}(E,M+1)$.
Outside the union of the two exceptional events,
\begin{align*}
\norm{F_Z}_\infty \le\norm{P_Z}_\infty+\norm{Q_Z}_\infty+\norm{R_Z}_\infty \le\frac{C_{\ref{prop:low}}+2}{N\sqrt n}.
\end{align*}
The union bound gives exceptional probability at most
$2e^{-(M+1)n}\le e^{-Mn}$. This proves the proposition.
\end{proof}

\subsection{High moments and approximate relations}
\label{sec:highmoment}

Write $\mathfrak m=\{\theta:|\theta|>n^{2}/N\}$. For $r\ge1$, define
$H_r=\{z\in\R^r:\sum_jz_j=0\}$, with the measure obtained by
eliminating one coordinate. In the chart
$z_r=-\sum_{j<r}z_j$, this is $dz_1\cdots dz_{r-1}$.
This measure is invariant under coordinate permutations. In dimension zero, integration means evaluation at the unique point.

The following lemma converts the supremum bound into a high moment.
The moment is taken before absolute values are introduced into the
Fourier integral, so the relations between the frequencies are retained.

\begin{lemma}\label{lem:moment}
For each coefficient density $\mu_i$, there is a probability
density $\nu_i$ such that
\begin{align}
\mu_i\le2\nu_i,\qquad
 |\widehat\nu_i(t)|\le e^{-aN^2t^2}\quad(t\in\R),
 \label{eq:domination}
\end{align}
with an absolute constant $a>0$. Define
\begin{align*}
K_i(\theta)=q^{-k}\sum_{v\in E^k}\widehat\nu_i(-\langle v,\theta\rangle).
\end{align*}
There is a constant $C_{\ref{lem:moment}}>0$ depending only on $E$
such that for every even integer $k\ge2$,
\begin{align}
\PP_Z\left(\norm{R_Z}_\infty>\frac1{N\sqrt n}\right) \le 2^n C_{\ref{lem:moment}}^k N^{k+1}n^{(k-1)/2} \int_{H_k\cap\mathfrak m^k} \prod_{j=1}^k\omega_n(\theta_j) \prod_{i=1}^n|K_i(\theta)|\,d\theta.
 \label{eq:normalizedmoment}
\end{align}
\end{lemma}

\begin{proof}
Conjugate symmetry makes $R_z$ real valued. Differentiation under
the integrable Gaussian gives
\begin{align*}
\norm{R'_z}_\infty
 \le\int_\R 2\pi|\theta|e^{-2\pi^2\beta^2n\theta^2}\,d\theta
 =\frac1{\pi\beta^2n}.
\end{align*}
Suppose
$\norm{R_z}_\infty>\lambda$ and choose $x_0$ with
$|R_z(x_0)|>\lambda$. The derivative bound implies
$|R_z(x)|\ge\lambda/2$ whenever
$|x-x_0|\le cn\lambda$. Thus, for every even $k\ge2$,
\begin{align*}
\int_{\R}R_z(x)^k\,dx
 \ge cn\lambda\left(\frac{\lambda}{2}\right)^k.
\end{align*}
Markov's inequality therefore gives
\begin{align*}
\PP\left(\norm{R_Z}_\infty>\lambda\right)
 \le\frac{C2^k}{n\lambda^{k+1}}
                  \EE\int_{\R}R_Z(x)^k\,dx.
\end{align*}

For $g_z=\one_{\mathfrak m}\omega_n\Phi_z$, we have
$g_z\in L^1\cap L^2$ and $g_z(-\theta)=\overline{g_z(\theta)}$.
Writing $k=2\ell$, Plancherel applied to $R_z^\ell$ and the
convolution of $\ell$ copies of $g_z$ gives
\begin{align*}
\int_{\R}R_z(x)^k\,dx
 =g_z^{*k}(0)
 =\int_{H_k\cap\mathfrak m^k}
     \prod_{j=1}^k\omega_n(\theta_j)
     \prod_{i=1}^n\prod_{j=1}^k\varphi(z_i\theta_j)\,d\theta.
\end{align*}
The constraint $\sum_j\theta_j=0$ accounts for $H_k$; convolution
uses exactly its measure defined by eliminating one coordinate, with no additional
normalizing factor. We now introduce auxiliary densities solely to
average this nonnegative spatial moment. We construct the densities $\nu_i$ as follows.

For an interval of length $L$, convolve the uniform density on the
concentric interval of length $3L/2$ with the density of a centered
Gaussian $G_L$ of standard deviation $L/16$. On the original
interval, the resulting density is at least
\begin{align*}
\frac{2}{3L}\PP\left(|G_L|\le L/4\right)\ge\frac1{2L}.
\end{align*}
Its Fourier transform has modulus at most $e^{-aL^2t^2}$.
For a union of two intervals of lengths $L_1,L_2$, mix the
corresponding densities with weights $L_j/(L_1+L_2)$, $j=1,2$.
This proves \eqref{eq:domination}.

Since $k$ is even,
$z\mapsto\int_{\R}R_z(x)^k\,dx$ is nonnegative. Hence
\begin{align*}
\EE_\mu\int_{\R}R_Z(x)^k\,dx =\int_{\R^n}\left(\int_{\R}R_z(x)^k\,dx\right) \prod_{i=1}^n\mu_i(z_i)\,dz \le 2^n\int_{\R^n}\left(\int_{\R}R_z(x)^k\,dx\right) \prod_{i=1}^n\nu_i(z_i)\,dz
\end{align*}
by \eqref{eq:domination}. Let $\widetilde Z_i$ be independent with
densities $\nu_i$. Expanding the characteristic functions before
integrating in $\widetilde Z_i$ gives
\begin{align*}
\EE\prod_{j=1}^k\varphi(\widetilde Z_i\theta_j) =q^{-k}\sum_{(v_1,\ldots,v_k)\in E^k} \EE e^{2\pi i\widetilde Z_i(v_1\theta_1+\cdots+v_k\theta_k)} =q^{-k}\sum_{v\in E^k}\widehat\nu_i(-\langle v,\theta\rangle) =:K_i(\theta).
\end{align*}
The Gaussian product is integrable on $H_k$, so Fubini justifies these
exchanges before absolute values are taken. Finally, at
$\lambda=(N\sqrt n)^{-1}$,
\begin{align*}
\frac1{n\lambda^{k+1}}=N^{k+1}n^{(k-1)/2}.
\end{align*}
Substitution in the preceding supremum estimate gives
\eqref{eq:normalizedmoment}.
\end{proof}
We now express the kernel bounds in terms of approximate relations.
With $a$ as in \eqref{eq:domination}, set
\begin{align*}
R=\sqrt{3k\log q/a},\qquad
V(\theta)=\{v\in E^k:|\langle v,\theta\rangle|\le R/N\}.
\end{align*}
For a uniform vector $b\in E^k$,
\begin{align*}
q^{-k}|V(\theta)|
 =\PP_b\left(|\langle b,\theta\rangle|\le R/N\right).
\end{align*}
Equation~\eqref{eq:domination} and $e^{-aR^2}=q^{-3k}$ give
\begin{align}
|K_i(\theta)|
 \le q^{-k}\sum_{v\in V(\theta)}
       e^{-aN^2\langle v,\theta\rangle^2}+q^{-3k}
 \le q^{-k}|V(\theta)|+q^{-3k}.
 \label{eq:Krelations}
\end{align}
We retain both bounds: the cardinality controls most factors,
while the Gaussian weights will be used to integrate the remaining
factors.

To estimate the frequency integral, we classify frequencies by
the dimension $r$ of a suitable subspace $U$ and the relative
density $|V(\theta)|/q^r$. For a fixed basis of $U$ chosen from
$V(\theta)$, the defining inner-product bounds restrict $\theta$
to an intersection of slabs; quantitative control of the basis
bounds its widths along $U$. We combine subspace counting with
integration over these regions, in analogy with small ball
estimates and nets in Section~5.2 of \cite{RV}.
For general real supports, a basis of $\spn V(\theta)$
chosen from $V(\theta)$ may be poorly conditioned. The next lemma selects a subspace
approximating $V(\theta)$, with an injective coordinate projection
and quantitative control of a projected basis.

\begin{lemma}\label{lem:relations}
Let $k\ge2$, $N\ge1$, and fix $\tilde{a}\in E\setminus\{0\}$.
There are constants
$C_{\ref{lem:relations}},c_{\ref{lem:relations}}>0$ depending only
on $E$ such that the following holds. Set
\begin{align*}
\eta=\exp[-C_{\ref{lem:relations}}\log^2(k+1)].
\end{align*}
For each $\theta\in H_k$, there is a subspace $U=U(\theta)$ of
dimension $1\le r\le k$, spanned by $r$ vectors from $V(\theta)$,
with the following properties:
\begin{enumerate}[label=(\roman*),leftmargin=2em]
\item $\tilde{a}\one_k\in U$ and $\dist(v,U)\le\eta$ for all
$v\in V(\theta)$.
\item After a coordinate permutation,
$U=\{(x,Gx):x\in\R^r\}$, where $|G_{ij}|\le1$ and
$G\one_r=\one_{k-r}$. The input projection is injective on
$V(\theta)$. In particular, we can write $V(\theta)=\{(x,f(x)):x\in\mathcal S\}$, where
$\mathcal S\subset E^r$ contains $\tilde{a}\one_r$, and
\begin{align*}
\norm{f(x)-Gx}_\infty\le\sqrt{k+1}\,\eta\qquad(x\in\mathcal S).
\end{align*}
There is a basis $Q$ with columns in $\mathcal S$, whose corresponding
vectors in $V(\theta)$ form a basis of $U$, such that
\begin{align*}
s_{\min}(Q)\ge c_{\ref{lem:relations}}\eta/k^3.
\end{align*}
\end{enumerate}
The choices can be made measurably in $\theta$.
\end{lemma}

\begin{proof}
Fix $\theta\in H_k$ and write $V=V(\theta)$.
Let $P$ be the orthogonal projection onto $H_k$.
Choose $C_{\ref{lem:relations}}$ larger than the precision
constant in Proposition~\ref{prop:approxcount} by a sufficiently
large fixed amount. For $0\le j\le k-1$, let $v_j^*$ be the largest
volume of a frame of $j$ vectors from $PV$, with $v_0^*=1$.
Set
\begin{align*}
p=\min\operatorname*{arg\,max}_{0\le j\le k-1}\eta^{-j}v_j^*
\end{align*}
and choose a maximizing frame of $p$ vectors
$(Pv_1,\ldots,Pv_p)$. Define
\begin{align*}
r=p+1,\qquad d=k-r,\qquad
 U=\spn(\tilde{a}\one_k,v_1,\ldots,v_p).
\end{align*}
Ties throughout are resolved by fixed orderings of the finite product sets. These choices are consequently measurable.

Write $Q_0=(Pv_1,\ldots,Pv_p)$ and $\Gamma=Q_0^TQ_0$ when $p>0$.
Comparison of the maximizing value with dimension $p-1$ gives
\begin{align*}
\eta^{-(p-1)}v_{p-1}^*\le\eta^{-p}v_p^*,
 \qquad \frac{v_{p-1}^*}{v_p^*}\le\eta^{-1}.
\end{align*}
Let $Q_{0,\widehat i}$ be $Q_0$ with column $i$ deleted. The cofactor formula gives
\begin{align*}
(\Gamma^{-1})_{ii}
 =\frac{\vol(Q_{0,\widehat i})^2}{(v_p^*)^2}
 \le\eta^{-2}.
\end{align*}
Thus $\norm{\Gamma^{-1}}\le\operatorname{tr}\Gamma^{-1}\le p\eta^{-2}$,
and $s_{\min}(Q_0)\ge\eta/\sqrt p$.
If $p<k-1$, comparison with dimension $p+1$ gives
\begin{align*}
\eta^{-(p+1)}v_{p+1}^*\le\eta^{-p}v_p^*,
 \qquad \frac{v_{p+1}^*}{v_p^*}\le\eta.
\end{align*}
Appending $Pv$ to $Q_0$ multiplies volume by its distance to the
span of $Q_0$. Consequently
\begin{align*}
\dist(v,U)=\dist(Pv,\spn(Q_0))
 \le v_{p+1}^*/v_p^*\le\eta\qquad(v\in V).
\end{align*}
For $p=0$ this follows from $v_0^*=1$ and $v_1^*\le\eta$;
there is no inverse Gram matrix to estimate. For $p=k-1$,
$U=\R^k$ and the distance assertion is immediate; no frame of $p+1$ vectors
is required.

Choose a maximal coordinate minor on $U$ and denote its input
coordinates by $I$, $|I|=r$. After a permutation,
$U=\{(x,Gx):x\in\R^r\}$ with $G\in\R^{d\times r}$ and $|G_{ij}|\le1$.
To see this, take any $k\times r$ basis matrix of $U$ and
choose its $r$ input rows to maximize the absolute determinant.
Replacing its $j$th input row by an output row multiplies this
determinant by the corresponding entry of $G$. This proves the bound. Since $\tilde{a}\one_k\in U$, we also have
$G\one_r=\one_d$.
If $v=(x,y)\in V$ and $v-w=e$ with $w\in U$ and
$\norm e_2\le\eta$, then $y-Gx=e_{I^c}-Ge_I$.
Each row of $(-G,I_d)$ has norm at most $\sqrt{r+1}$, so
\begin{align*}
\norm{y-Gx}_\infty\le\sqrt{k+1}\,\eta.
\end{align*}
We choose $C_{\ref{lem:relations}}$ large enough that twice this bound is less
than $\Delta$. If two vectors in $V$ have the same input $x$, their output
coordinates differ by at most $2\sqrt{k+1}\eta<\Delta$.
Since both outputs are in $E^d$, they agree. Thus projection onto
$I$ is injective on $V$.
Write its image as $\mathcal S\subset E^r$ and its output labels
as $f:\mathcal S\to E^d$. Thus $|\mathcal S|=|V|\le q^r$.
The selected frame lies exactly in $U$, so its projection is a
basis $Q$ contained in $\mathcal S$; in particular $\mathcal S$
spans $\R^r$ and contains $\tilde{a}\one_r$.

We also have the quantitative bound
\begin{align*}
s_{\min}(Q)\ge c\eta/k^3.
\end{align*}
To verify the bound for $p>0$, let
$Q_1=(\tilde{a}\one_k,v_1,\ldots,v_p)$ and
$y=Q_1(\alpha_0,\alpha)$, where $\alpha\in\R^p$.
Since $Py=Q_0\alpha$ and $s_{\min}(Q_0)\ge\eta/\sqrt p$,
\begin{align*}
\norm{\alpha}_2
 \le\sqrt p\,\eta^{-1}\norm{Py}_2
 \le\sqrt p\,\eta^{-1}\norm y_2.
\end{align*}
Projection onto $\spn(\one_k)$ gives
$(I-P)y=\alpha_0\tilde{a}\one_k+
\sum_{i=1}^p\alpha_i(I-P)v_i$.
Since $v_i\in E^k$, we have $\norm{v_i}_2\le C\sqrt k$.
The triangle inequality and Cauchy--Schwarz therefore give
\begin{align*}
|\tilde{a}|\sqrt k\,|\alpha_0| \le\norm{(I-P)y}_2
      +\sum_{i=1}^p|\alpha_i|\norm{(I-P)v_i}_2
 \le\norm y_2+C\sqrt{kp}\norm{\alpha}_2.
\end{align*}
It follows that $s_{\min}(Q_1)\ge c\eta/k^2$, where $c$
depends only on $E$. The identity $Q_1=\binom{I_r}{G}Q$
and the bound $\norm{\binom{I_r}{G}}\le\sqrt{1+dr}\le k+1$
give the claimed bound on $s_{\min}(Q)$, after changing $c$.
For $p=0$ the input basis is $(\tilde{a})$ and the assertion follows
directly.
\end{proof}

\subsection{Gaussian integration and localization}
\label{sec:auxiliary}
We use two auxiliary estimates for bounded density loss.
The first controls the Gaussian integral arising from
\eqref{eq:Krelations}.
The second bounds frequency vectors satisfying many
approximate relations and replaces sufficiently accurate
approximate interpolation by exact interpolation.

\begin{lemma}\label{lem:gramintegral}
Fix a finite $E\subset\R$ with $|E|=q\ge2$ and $a>0$.
There are constants
$C_{\ref{lem:gramintegral}}=C_{\ref{lem:gramintegral}}(E,a)>0$ and
$c_{\ref{lem:gramintegral}}=c_{\ref{lem:gramintegral}}(E)>0$.
Set
\begin{align*}
\Psi_{N,r}(z)=q^{-r}\sum_{b\in E^r}e^{-aN^2(\langle b,z\rangle)^2}.
\end{align*}
For $N,L\ge1$ and positive integers $r,s$ with $r-1\le s/2$,
\begin{align*}
\int_{\substack{z\in H_r\\|z_j|\le L/N}}\Psi_{N,r}(z)^s\,dz
 \le\left[\frac{C_{\ref{lem:gramintegral}}}{N}
             \left(s^{-1/2}+Le^{-c_{\ref{lem:gramintegral}}s}\right)\right]^{r-1}.
\end{align*}
\end{lemma}

\begin{proof}
The assertion is immediate for $r=1$. Assume $r\ge2$ and set
$p=r-1$. Let $B$ be an $s\times r$ matrix with independent
entries uniformly distributed on $E$, and define
\begin{align*}
D=(B_{\cdot1}-B_{\cdot r},\ldots,B_{\cdot p}-B_{\cdot r}).
\end{align*}
For $z\in H_r$, set $u_j=Nz_j$ for $j\le p$.
Since $z_r=-\sum_{j=1}^p z_j$, we have $NBz=Du$, and
independence of the rows of $B$ gives
\begin{align*}
\Psi_{N,r}(z)^s=\EE_B e^{-a\norm{Du}_2^2}.
\end{align*}
Using the coordinate representation of the measure on $H_r$
and Tonelli's theorem, we obtain
\begin{align*}
\int_{\substack{z\in H_r\\ |z_j|\le L/N}}
     \Psi_{N,r}(z)^s\,dz
 \le N^{-p}\EE_B\int_{[-L,L]^p}e^{-a\norm{Du}_2^2}\,du.
\end{align*}
On this box, $\norm u_2^2/L^2\le p$. Thus Gaussian
integration gives
\begin{align*}
\int_{[-L,L]^p}e^{-a\norm{Du}_2^2}\,du
 \le e^p\int_{\R^p}
       e^{-a\norm{Du}_2^2-\norm u_2^2/L^2}\,du =e^p\pi^{p/2}\det(aD^TD+L^{-2}I)^{-1/2}.
\end{align*}

Conditional on $B_{\cdot r}$, the columns of $D$ are
independent product vectors with a common deterministic shift.
We therefore prove a negative determinant moment bound uniform
over such shifts. Let $Z=(w_1,\ldots,w_p)$, where
$w_j=b_j-v_j\in\R^s$, the vectors $b_j$ are independent and
uniform on $E^s$, and the $v_j$ are deterministic.

For $\lambda>0$, let $Z_{<j}=(w_1,\ldots,w_{j-1})$.
The successive Schur complements satisfy
\begin{align*}
    \det(aZ^TZ+\lambda I)=\prod_{j=1}^p S_j,
\end{align*}
\begin{align*}
S_j=\lambda+aw_j^T
   \left[I-aZ_{<j}
   (aZ_{<j}^TZ_{<j}+\lambda I)^{-1}Z_{<j}^T\right]w_j \ge\lambda+a\norm{P_jw_j}_2^2,
\end{align*}
where $P_j$ is the orthogonal projection onto the complement
of the span of the earlier columns. Indeed, the matrix in
brackets is positive semidefinite and equals the identity on
that complement. Conditional on $Z_{<j}$, the projection
$P_j$ is fixed and $h_j:=\rank P_j\ge s-j+1\ge s/2$.

Applying Lemma~\ref{lem:projection}, we have
\begin{align*}
\PP\left(
 \norm{P_jw_j}_2^2<c_{\ref{lem:projection}}s/2
 \,\middle|\,Z_{<j}\right)
 \le2e^{-c'_{\ref{lem:projection}}s/2}.
\end{align*}
On the complementary event, $S_j^{-1/2}\le Cs^{-1/2}$,
where $C=C(E,a)$, while always
$S_j^{-1/2}\le\lambda^{-1/2}$. Hence, with
$c=c'_{\ref{lem:projection}}/2$,
\begin{align*}
\EE(S_j^{-1/2}\mid Z_{<j})
 \le C\left(s^{-1/2}+\lambda^{-1/2}e^{-cs}\right).
\end{align*}
Since $S_1,\ldots,S_{p-1}$ depend only on $Z_{<p}$,
successive conditioning, starting with the last column, yields
\begin{align*}
\EE\det(aZ^TZ+\lambda I)^{-1/2} =\EE\left[
   \prod_{j=1}^{p-1}S_j^{-1/2}
   \EE(S_p^{-1/2}\mid Z_{<p})\right]\le\left[
   C\left(s^{-1/2}+\lambda^{-1/2}e^{-cs}\right)
   \right]^p.
\end{align*}

Apply this bound to $D$ conditional on $B_{\cdot r}$, with
$\lambda=L^{-2}$, and then average over $B_{\cdot r}$.
Substituting into the integral estimate and increasing $C$
gives
\begin{align*}
\int_{\substack{z\in H_r\\ |z_j|\le L/N}}
     \Psi_{N,r}(z)^s\,dz
 \le\left[
   \frac{C}{N}\left(s^{-1/2}+Le^{-cs}\right)
   \right]^{r-1}.
\end{align*}
\end{proof}

For bounded density loss, scalar anticoncentration bounds
the number of coefficients above the approximation scale.
This reduces the argument to finitely many interpolation
problems in bounded dimension.
The next lemma uses this observation to bound the frequency
vector and to replace sufficiently accurate approximate
interpolation by exact interpolation.
Neither conclusion requires determinant bounds in growing
dimension.

\begin{lemma}\label{lem:localization}
Fix $T\ge0$. There are constants
$C_{\ref{lem:localization}}=C_{\ref{lem:localization}}(E,T)>0$ and
$c_{\ref{lem:localization}}=c_{\ref{lem:localization}}(E,T)>0$ such that:

\begin{enumerate}[label=(\roman*),leftmargin=2em]
\item If $\mathcal S\subset E^r$ spans $\R^r$,
$|\mathcal S|>q^{r-T-1}$ and $|\langle x,z\rangle|\le\rho$ for every
$x\in\mathcal S$, then
\begin{align*}
\norm z_\infty\le C_{\ref{lem:localization}}r\rho.
\end{align*}
\item If $\mathcal S\subset E^r$ satisfies
$|\mathcal S|>q^{r-T-1}$ and
$\norm{Gx-f(x)}_\infty\le\rho
 \le c_{\ref{lem:localization}}/r$
for every $x\in\mathcal S$, where
$G\in\R^{d\times r}$ and $f:\mathcal S\to E^d$,
then there is $G'\in\R^{d\times r}$ with
$f(x)=G'x$ for all $x\in\mathcal S$.
\end{enumerate}
\end{lemma}

\begin{proof}
For (i), the case $\rho=0$ follows from the spanning assumption. Assume $\rho>0$ and define $J=\{j:|z_j|>2\rho/\Delta\}$.

If $J$ is empty, the conclusion holds directly. Otherwise,
since every point of $\mathcal S$ satisfies the small inner
product condition, Lemma~\ref{lem:scalar} gives
\begin{align*}
q^{-T-1}
 <q^{-r}|\mathcal S|
 \le\PP\left(|\langle b,z\rangle|\le\rho\right)
 \le\frac{C}{\sqrt{|J|}}.
\end{align*}
Hence $|J|\le Cq^{2T+2}$.

The projection of $\mathcal S$ onto $J$ spans $\R^J$.
Choose $|J|$ points of $\mathcal S$ whose projections form
the rows of an invertible matrix $B$. For each selected point,
the coordinates outside $J$ contribute at most $|\langle x_{J^c},z_{J^c}\rangle|
 \le C_Er\rho$.
 
Thus $\norm{Bz_J}_\infty\le C_Er\rho$.
Since the size of $B$ is bounded in terms of $E,T$ and its
entries lie in $E$, there are only finitely many possible
invertible matrices $B$. Their inverse norms are uniformly
bounded in terms of $E,T$, giving
$\norm{z_J}_\infty\le C_{E,T}r\rho$.
Together with $|z_j|\le2\rho/\Delta$ outside $J$, this proves (i).

For (ii), it suffices to treat one row $a$ of $G$, with
corresponding label function $f_0:\mathcal S\to E$.
The case $\rho=0$ is immediate. Assume $\rho>0$ and define $J=\{j:|a_j|>2\rho/\Delta\}$.

If $J$ is nonempty, a union bound over the $q$ possible
labels and Lemma~\ref{lem:scalar} give
\begin{align*}
q^{-T-1}
 <q^{-r}|\mathcal S|
 \le\PP\left(\dist(\langle a,b\rangle,E)\le\rho\right)
 \le\frac{qC}{\sqrt{|J|}}.
\end{align*}
Consequently, $|J|\le Cq^{2T+4}$, including the empty case.

We next show that the label is determined by the coordinates
in $J$. If $x,y\in\mathcal S$ satisfy $x_J=y_J$, then
\begin{align*}
|f_0(x)-f_0(y)| \le2\rho+|\langle a,x-y\rangle| =2\rho+|\langle a_{J^c},x_{J^c}-y_{J^c}\rangle|
 \le C_Er\rho.
\end{align*}
Since $\rho\le c_{\ref{lem:localization}}/r$, choosing
$c_{\ref{lem:localization}}$ sufficiently small makes this
less than $\Delta$. Both labels lie in $E$, so they must
be equal. Thus, with $F=\pi_J(\mathcal S)$, we may define
$\widetilde f_0:F\to E$ by
\begin{align*}
\widetilde f_0(x_J)=f_0(x).
\end{align*}
The omitted coordinates also give
\begin{align*}
|\langle a_J,y\rangle-\widetilde f_0(y)|
 \le C_Er\rho
 \qquad(y\in F).
\end{align*}

Let $m_0=m_0(E,T)$ be an integer upper bound for $|J|$.
For each $F\subset E^j$, $0\le j\le m_0$, the evaluations
of linear forms on $F$ constitute the closed linear subspace
\begin{align*}
\mathcal L_F
 =\{(\langle c,y\rangle)_{y\in F}:c\in\R^j\}
 \subset\R^F.
\end{align*}
There are only finitely many pairs $(F,g)$ with
$F\subset E^j$, $0\le j\le m_0$, and $g:F\to E$.
For every such pair with $g\notin\mathcal L_F$, the distance
$\dist_\infty(g,\mathcal L_F)$ is positive.
Let $\varepsilon=\varepsilon(E,T)>0$ be the minimum of
these positive distances, setting $\varepsilon=1$ if there
are no such pairs.

For our projected labels,
\begin{align*}
\dist_\infty(\widetilde f_0,\mathcal L_F)
 \le C_Er\rho
 \le C_Ec_{\ref{lem:localization}}.
\end{align*}
Decreasing $c_{\ref{lem:localization}}$ so that the last
quantity is less than $\varepsilon$ forces
$\widetilde f_0\in\mathcal L_F$.
Hence there is a linear form on $\R^J$ taking the values
$\widetilde f_0$ on $F$. Extending its coefficients by zero
outside $J$ gives a row $a'\in\R^r$ satisfying
\begin{align*}
\langle a',x\rangle=f_0(x)
\qquad(x\in\mathcal S).
\end{align*}
Applying this construction to each output row gives $G'$.
The same choice of $c_{\ref{lem:localization}}$ works for
every row and depends only on $E,T$, independently of $r,d$.
\end{proof}

\subsection{Proof of the far frequency estimate}
\label{sec:farproof}

\begin{proof}[Proof of Proposition~\ref{prop:far}]
Fix $M\ge1$ and $D>1$. We first choose a fixed $A\ge1$, then a
fixed integer $T\ge1$, and finally take
$n\ge n_{\ref{prop:far}}(E,M,D)$.
The choices of $A,T$ and this dimension threshold will be made
at the end of the proof.
Use the even moment order
\begin{align*}
k=2\left\lfloor\frac{An}{2\log(n+1)}\right\rfloor
  =\frac{An}{\log(n+1)}+O(1).
\end{align*}
Use $R=\sqrt{3k\log q/a}$ and the parameter $\eta$ from Lemma~\ref{lem:relations}. Since $A$ and $T$ are fixed independently of $n$,
\begin{align*}
2C_{\ref{lem:localization}}kR
=O_{E,A,T}\left(
\left(\frac{n}{\log(n+1)}\right)^{3/2}
\right)
=o(n^2).
\end{align*}
After increasing $n_{\ref{prop:far}}(E,M,D)$, we may therefore assume throughout that
\begin{align*}
n^2\ge 2C_{\ref{lem:localization}}kR.
\end{align*}

For each $\theta\in H_k\cap\mathfrak m^k$, select $U(\theta)$ by
Lemma~\ref{lem:relations}, and set $r=\dim U$ and $d=k-r$.
The input projection of $V$ is injective, so $|V|\le q^r$;
also $V$ contains all $q$ constant vectors. There is therefore a
unique integer $0\le t\le r-1$ such that
\begin{align*}
q^{-t-1}<|V(\theta)|/q^r\le q^{-t}.
\end{align*}
By \eqref{eq:Krelations} and the definition of $d,t$, with $\varepsilon=q^{-k-1}$,
\begin{align}
|K_i(\theta)|\le q^{-d-t}+q^{-3k}
                    \le(1+\varepsilon)q^{-d-t}.
 \label{eq:Kpointwise}
\end{align}

Define
\begin{align*}
\Omega_{d,t}
 =\left\{\theta\in H_k\cap\mathfrak m^k:
 \dim U(\theta)=k-d,\quad
 q^{k-d-t-1}<|V(\theta)|\le q^{k-d-t}\right\},
\end{align*}
and set
\begin{align*}
\mathcal T_{d,t}
 =2^n C_{\ref{lem:moment}}^kN^{k+1}n^{(k-1)/2}
 \int_{\Omega_{d,t}}\prod_{j=1}^k\omega_n(\theta_j)
                       \prod_{i=1}^n|K_i(\theta)|\,d\theta.
\end{align*}
The sets $\Omega_{d,t}$ form a measurable partition. Lemma~\ref{lem:moment}
gives
\begin{align*}
\PP_Z\left(\norm{R_Z}_\infty>\frac1{N\sqrt n}\right)
 \le\sum_{d=0}^{k-1}\sum_{t=0}^{k-d-1}\mathcal T_{d,t}.
\end{align*}
There are at most $k^2$ summands. We shall obtain, uniformly in
$d,t,N$ for $n\ge n_{\ref{prop:far}}(E,M,D)$,
\begin{align*}
\log\mathcal T_{d,t}\le
\begin{cases}
 [\log(2q^2)-A(D-1)/2+o(1)]n,&t\le T,\\
 C(A+1)n-nt\log q/2,&t>T.
\end{cases}
\end{align*}
For $t\le T$, we show that $U=\spn V$, count the possible
spaces $U$, and estimate the integral over each corresponding
region. For $t>T$, we use a weighted cover by relation slabs,
with the resulting counting and integration costs absorbed
by the factor $q^{-nt}$.

\smallskip
\noindent\textit{The case $t\le T$.}
We first show that all approximate relation vectors lie exactly in the selected subspace $U$. In the
graph representation of Lemma~\ref{lem:relations}, the input set
$\mathcal S$ has size $|V|>q^{r-T-1}$ and interpolation error at
most $\sqrt{k+1}\eta$. Since $k\sqrt{k+1}\eta\to0$,
Lemma~\ref{lem:localization}(ii) gives, for all sufficiently large
$k$, an exact linear graph containing $V$. This graph and $U$ both
have dimension $r$ and contain the selected $r$ independent vectors
from $V$. They coincide, so
\begin{align*}
U=\spn V,\qquad V\subset U,\qquad |U\cap E^k|>q^{k-d-t-1}.
\end{align*}
There are at most $\exp \mathcal H_d(k,t+1)$ possible spaces $U$.

We partition the integral according to the value of $U(\theta)$.
For each such space, define
\begin{align*}
\Omega_{d,t}(U)=\{\theta\in\Omega_{d,t}:U(\theta)=U\}.
\end{align*}
These regions are disjoint, and hence
\begin{align}
\mathcal T_{d,t} ={}2^n C_{\ref{lem:moment}}^kN^{k+1}n^{(k-1)/2} \sum_U\int_{\Omega_{d,t}(U)} \prod_{j=1}^k\omega_n(\theta_j) \prod_{i=1}^n|K_i(\theta)|\,d\theta.
 \label{eq:sumoverspaces}
\end{align}
We estimate one integral in this sum. Fix $U$ with
$\Omega_{d,t}(U)\ne\emptyset$; all the following kernel estimates
are first established on that region.

Choose $r$ graph coordinates for $U$, writing
after permutation
\begin{align*}
U=\{(b,Gb):b\in\R^r\},\qquad G\one_r=\one_d.
\end{align*}
For these graph coordinates use
\begin{align*}
\theta=(z-G^T\xi,\xi),\qquad z\in H_r,\quad\xi\in\R^d.
\end{align*}
Then $\langle(b,Gb),\theta\rangle=\langle b,z\rangle$.
Thus the relation constraints involve only $z\in H_r$,
while $(-G^T\xi,\xi)$ parametrizes $U^\perp$.
Since $G\one_r=\one_d$,
\begin{align*}
\langle\one_k,\theta\rangle
 =\langle\one_r,z\rangle
  +\langle\one_d-G\one_r,\xi\rangle=0.
\end{align*}
The change of variables is therefore a linear isomorphism
from $H_r\times\R^d$ onto $H_k$.

Use the coordinate measures on $H_r$ and $H_k$ obtained
by eliminating the same input coordinate, and Lebesgue
measure on $\R^d$. The resulting transformation matrix
is block triangular with identity diagonal blocks,
so the Jacobian is one.

The projection $S_0$ of $V$ onto the input coordinates spans
$\R^r$, has $|S_0|=|V|$, and obeys $|\langle b,z\rangle|\le R/N$.
For $t\le T$, Lemma~\ref{lem:localization}(i) implies
\begin{align*}
|z_j|\le C_{\ref{lem:localization}}rR/N \le n^{2}/N.
\end{align*}
If a coordinate vector $e_j$ belongs to $U$, any injective graph
projection must retain its coordinate, and the corresponding column
of $G$ is zero. Thus $\theta_j=z_j$, contrary to $\theta \in \mathfrak{m}^{k}$. Hence a space $U$ with
$\Omega_{d,t}(U)\ne\emptyset$ contains no coordinate vector. Applying \eqref{eq:codim} with $u=T+1$ to obtain
\begin{align*}
d\ge k/2-C,
\end{align*}
where $C>0$ depends only on $q$ and $T$.

We now estimate the integral for this fixed $U$. In the first bound
of \eqref{eq:Krelations}, every $v\in V$ has $\langle v,\theta\rangle=\langle b,z\rangle$.
Enlarging the projected sum to all $b\in E^r$ gives
\begin{align*}
|K_i(\theta)|\le q^{-d}\Psi_{N,r}(z)+q^{-3k}.
\end{align*}
Since $z\in H_r$, constant configurations give
$\Psi_{N,r}(z)\ge q^{1-r}$, and hence
\begin{align*}
|K_i(\theta)|\le(1+\varepsilon)q^{-d}\Psi_{N,r}(z).
\end{align*}
Take $s=2k$ and assume $n\ge s$. Apply the Gram integral estimate to $s$ of the $n$
kernel factors, using this bound;
\eqref{eq:Kpointwise} keeps the additional density saving
$q^{-(n-s)t}$ in the others. Retaining the $d$
Gaussian factors at the output coordinates integrates $\xi$ at a
cost $(C/\sqrt n)^d$. Lemma~\ref{lem:gramintegral}, with
$L=C_{\ref{lem:localization}}kR$, gives
\begin{align*}
\int_{\substack{z\in H_r\\|z_j|\le C_{\ref{lem:localization}}rR/N}}
             \Psi_{N,r}(z)^s\,dz
 \le\left(\frac{C}{N\sqrt{s}}\right)^{r-1}
\end{align*}
for all large $k$. Its required condition is
$C_{\ref{lem:localization}}kR e^{-cs}\le s^{-1/2}$, which holds for fixed $E,T$.

We have therefore proved the following bound for each nonempty
region $\Omega_{d,t}(U)$:
\begin{align*}
\int_{\Omega_{d,t}(U)} \prod_{j=1}^k\omega_n(\theta_j)\prod_{i=1}^n|K_i(\theta)|\,d\theta \le (1+\varepsilon)^nq^{-nd-(n-s)t} \left(\frac{C}{N\sqrt s}\right)^{r-1} \left(\frac{C}{\sqrt n}\right)^d.
\end{align*}
Here we enlarged the integration domain to the indicated box in
$H_r$ times all of $\R^d$ only after obtaining the kernel bounds
on $\Omega_{d,t}(U)$. All constants are uniform over these spaces. We can now insert this
bound into \eqref{eq:sumoverspaces} and use the count
$\exp \mathcal H_d(k,t+1)$.

Write $\Delta_n=\log(q^n/N)$. Combining the count, the two
integral factors and the moment normalization gives
\begin{align*}
\mathcal T_{d,t} \le [2q^2(1+\varepsilon)]^n C^k
 e^{\mathcal H_d(k,t+1)}(n/s)^{(r-1)/2}\cdot
 e^{-(d+2)\Delta_n-(n-s)t\log q},
 \qquad t\le T.
\end{align*}
Here $r=k-d$, so the powers of $N$ combine as
\begin{align*}
N^{k+1}N^{-(r-1)}q^{-nd}
 =q^{2n}e^{-(d+2)\Delta_n}.
\end{align*}

By \eqref{eq:refinedcount},
$\mathcal H_d(k,t+1)\le d\log(k+1)+O_{E,T}(k)$.
Using $\Delta_n\ge D\log(n+1)$,
$d\ge k/2-O_{E,T}(1)$ and
$k=An/\log(n+1)+O(1)$, we obtain
\begin{align}
\log\mathcal T_{d,t}
 &\le n\log(2q^2)
   +d\bigl[\log(k+1)-D\log(n+1)\bigr]+o(n)
   \notag\\
 &\le\left[\log(2q^2)-\frac{A(D-1)}2+o(1)\right]n.
 \label{eq:lowexponent}
\end{align}
The errors are uniform over $d$, $t\le T$ and the allowed $N$,
since $A,D,T$ are fixed, $k\log(n/s)=o(n)$ and
$\varepsilon=o(1)$.
The coefficient one in \eqref{eq:refinedcount} is precisely
what allows every fixed $D>1$.

\smallskip
\noindent\textit{The case $t>T$.}
In this range, the factor $q^{-nt}$ allows coarser counting
and integration estimates.
We apply Lemma~\ref{lem:densebases} to the coordinate
projection of $V$ and lift the resulting basis vectors
to $V$.
The corresponding intersections of slabs give a weighted
cover of $\Omega_{d,t}$.

Set $u=t+1$ and
\begin{align*}
m=(r-1-\lceil C_{\ref{lem:densebases}}(u+1)\rceil)_+,
\qquad g=r-1-m\le C(u+1).
\end{align*}
For each $\theta\in\Omega_{d,t}$, use $\mathcal S,f,Q$
from Lemma~\ref{lem:relations}. Since
$|\mathcal S|>q^{r-u}$ and $s_{\min}(Q)\ge c\eta/k^3$,
Lemma~\ref{lem:densebases} supplies at least
$e^{-Cr}|\mathcal S|^m\ge e^{-Cr}q^{(r-u)m}$
ordered bases $B_0=(b_1,\ldots,b_r)$ with columns in
$\mathcal S$ and $b_1=\tilde{a}\one_r$, satisfying
\eqref{eq:barycentric} and
\begin{align*}
|\det B_0|\ge \exp[-C_0k\log(k+1)]\eta^g,
\end{align*}
where $C_0$ is fixed using only $E$.

For each such basis, let
$Y=(f(b_1),\ldots,f(b_r))$ and $G_0=YB_0^{-1}$.
The identity $G_0-G=(Y-GB_0)B_0^{-1}$,
Lemma~\ref{lem:relations} and \eqref{eq:barycentric} give,
for every $x\in\mathcal S$,
\begin{align*}
\norm{G_0x-f(x)}_\infty
 \le\sqrt{k+1}\eta(1+\norm{B_0^{-1}x}_1) \le C(k+1)^5\eta
 \le\exp[-C'_{\ref{prop:approxcount}}\log^2(k+1)].
\end{align*}
The last inequality uses the fixed choice of
$C_{\ref{lem:relations}}$.

Let $\mathcal B_{r,t}$ consist of all triples $(I,B_0,Y)$
with $I\subset[k]$, $|I|=r$, an invertible
$B_0\in E^{r\times r}$ satisfying the preceding determinant
bound, and $Y\in E^{d\times r}$, such that
$(B_0)_{\cdot1}=\tilde{a}\one_r$,
$Y_{\cdot1}=\tilde{a}\one_d$, and
\begin{align*}
\left|\left\{x\in E^r:
\dist_\infty(YB_0^{-1}x,E^d)
 \le\exp[-C'_{\ref{prop:approxcount}}\log^2(k+1)]
\right\}\right|
 \ge q^{r-u}.
\end{align*}
For each triple, define $v_j\in E^k$ by
$(v_j)_I=(B_0)_{\cdot j}$ and
$(v_j)_{I^c}=Y_{\cdot j}$, using increasing coordinate order.
Then $v_1=\tilde{a}\one_k$. Define
\begin{align*}
\Lambda_{I,B_0,Y}
 =\{\theta\in H_k:
 |\langle v_j,\theta\rangle|\le R/N,\ 2\le j\le r\}.
\end{align*}

Proposition~\ref{prop:approxcount}, together with the
$2^k$ choices of $I$ and at most $q^{r(r-1)}$ choices
of $B_0$, gives
\begin{align*}
|\mathcal B_{r,t}|
 \le 2^kq^{r(r-1)}
 \exp\{C[k(u+1)+d\log(d+1)]\}.
\end{align*}
Each $\theta\in\Omega_{d,t}$ supplies at least
$e^{-Cr}q^{(r-u)m}$ distinct triples whose full vectors
belong to $V(\theta)$. Hence, on $H_k$,
\begin{align*}
\one_{\Omega_{d,t}}(\theta)
 \le e^{Cr}q^{-(r-u)m}
 \sum_{(I,B_0,Y)\in\mathcal B_{r,t}}
 \one_{\Lambda_{I,B_0,Y}}(\theta).
\end{align*}
Using
$r(r-1)-(r-u)m=rg+um\le Cr(u+1)$, we obtain the
total weight bound
\begin{align}
e^{Cr}q^{-(r-u)m}|\mathcal B_{r,t}|
 \le\exp\{Ck(t+2)+Cd\log(k+1)\}.
 \label{eq:fractionalcover}
\end{align}

Fix a triple in $\mathcal B_{r,t}$, permute coordinates
so that $I=[r]$, and set $G_0=YB_0^{-1}$.
Use the change of variables
\begin{align*}
\theta
 =\bigl(B_0^{-T}(0,y)-G_0^T\xi,\xi\bigr),
\qquad
(y,\xi)\in\R^{r-1}\times\R^d.
\end{align*}
The first-column conditions give $G_0\one_r=\one_d$ and
$B_0^{-T}(0,y)\in H_r$, so this is a linear isomorphism
onto $H_k$. Moreover,
$\langle v_j,\theta\rangle=y_{j-1}$ for $2\le j\le r$
and $\theta_{I^c}=\xi$. Thus $\Lambda_{I,B_0,Y}$
corresponds exactly to
\begin{align*}
[-R/N,R/N]^{r-1}\times\R^d.
\end{align*}

In the coordinate chart obtained by eliminating the last
input coordinate, the map $\theta\mapsto(y,\xi)$ has a
block triangular matrix, with an identity block in the
$\xi$ coordinates.
Subtracting the last row of $B_0$ from its first $r-1$ rows
shows that the other block has absolute determinant
$|\det B_0|/|\tilde{a}|$.
Hence the absolute Jacobian determinant of the displayed
change of variables is
\begin{align}
\frac{|\tilde{a}|}{|\det B_0|}
 \le\exp\{Ck\log(k+1)+C(t+2)\log^2(k+1)\}.
 \label{eq:realjacobian}
\end{align}
The $\log^2(k+1)$ term comes from $\eta^{-g}$, with
$g=O_E(t+2)$.

Retaining only the Gaussian factors at the free coordinates
gives, uniformly over $\mathcal B_{r,t}$,
\begin{align*}
\int_{\Lambda_{I,B_0,Y}}
 \prod_{j=1}^k\omega_n(\theta_j)\,d\theta
 \le \frac{|\tilde{a}|}{|\det B_0|}
       (2R/N)^{r-1}(C/\sqrt n)^d.
\end{align*}
Combining this estimate with \eqref{eq:fractionalcover},
\eqref{eq:realjacobian}, the kernel bound
\eqref{eq:Kpointwise}, and the moment normalization yields,
with $\Delta_n=\log(q^n/N)$,
\begin{align*}
\mathcal T_{d,t}\le{}&
 [2q^2(1+\varepsilon)]^n C^k
 (2R)^{r-1}n^{(r-1)/2}
 e^{-(d+2)\Delta_n-nt\log q}\\
&\quad\cdot
 \exp\{Ck(t+2)+Ck\log(k+1)
                    +C(t+2)\log^2(k+1)\}.
\end{align*}

Substituting $k=An/\log(n+1)+O(1)$ and $R=C\sqrt k$,
and using $\Delta_n\ge0$, gives
\begin{align}
\log\mathcal T_{d,t}
 \le C(A+1)n+Ck(t+2)
           +C(t+2)\log^2(k+1)-nt\log q.
 \label{eq:highexponent}
\end{align}
All these bounds are uniform in $d,t$ and the permitted $N$.

\smallskip
\noindent\textit{Choice of parameters and summation.}
Let $C=C(E)$ be the constant in \eqref{eq:highexponent}.
Choose $A\ge1$ and then a fixed integer $T\ge1$ such that
\begin{align*}
A(D-1)/2 > M+\log(2q^2)+10,\quad T\log q/2 > C(A+1)+M+6.
\end{align*}

With $A,T$ fixed, choose
$n_{\ref{prop:far}}(E,M,D)$ sufficiently large to satisfy
all preceding size conditions. Estimates
\eqref{eq:lowexponent} and \eqref{eq:highexponent} then give
\begin{align*}
\mathcal T_{d,t}\le e^{-(M+4)n},
\qquad n\ge n_{\ref{prop:far}}(E,M,D),
\end{align*}
uniformly in $d,t$ and the permitted $N$.

Summing over the at most $k^2$ pairs $(d,t)$ yields
\begin{align*}
\PP_Z\left(\norm{R_Z}_\infty>\frac1{N\sqrt n}\right)
 \le k^2e^{-(M+4)n}\le e^{-Mn}.
\end{align*}
All choices are uniform in $N$ and the coefficient intervals.
This proves Proposition~\ref{prop:far}.
\end{proof}

\section{Proof of the main theorem}\label{sec:normals}

We combine Theorem~\ref{thm:averaging} with the inversion of randomness
argument in Section~5 in \cite{Tik} and the structured estimate of
Jain, Sah and Sawhney \cite{JSS}. Throughout this section, $\xi$ is uniform on $S$, $q=|S|\ge2$, and
$\boldsymbol\xi=(\xi_1,\ldots,\xi_n)$ has independent coordinates
with this distribution.

For $\delta,\nu\in(0,1)$, define
\begin{align*}
\operatorname{Sparse}(\delta)
&=\{x\in\R^n:|\operatorname{supp}(x)|\le\delta n\},\\
\Comp(\delta,\nu)
&=\{x\in\mathbb S^{n-1}:
      \dist(x,\operatorname{Sparse}(\delta))\le\nu\},\\
\Incomp(\delta,\nu)
&=\mathbb S^{n-1}\setminus\Comp(\delta,\nu).
\end{align*}

\begin{lemma}\label{lem:structured}
There are $\delta_{\ref{lem:structured}}\in(0,1/4]$,
$\nu_{\ref{lem:structured}}\in(0,1/2]$ and
$\eta_{\ref{lem:structured}}>0$ depending only on $S$ such that,
with $\delta=\delta_{\ref{lem:structured}}$ and
$\nu=\nu_{\ref{lem:structured}}$, for
$n\ge n_{\ref{lem:structured}}(S)$ and $0\le t\le1$,
\begin{align}
\PP\left(\inf_{x\in\Comp(\delta,\nu)}\norm{M_nx}_2\le t\right)
 \le \tfrac12a_n(S)+(t+q^{-n})e^{-\eta_{\ref{lem:structured}}n}.
 \label{eq:structured}
\end{align}
\end{lemma}

\begin{proof}
Apply the estimate for almost constant vectors in Theorem 1.5 in \cite{JSS},
and denote its parameters by $\delta_0,\rho_0$. If
$x\in\Comp(\delta,\nu)$, at most
$(\delta+\nu^2/\rho_0^2)n$ coordinates have absolute value greater
than $\rho_0/\sqrt n$. Choose
$\delta\le\delta_0/2$ and $\nu^2\le\delta_0\rho_0^2/2$.
Then $x$ belongs to the class in that theorem, with common value zero.
Its elementary contribution is $a_n(S)/2$, giving
\eqref{eq:structured}.
\end{proof}

We first record a randomized rounding lemma that preserves the
small ball estimates needed below.
Random rounding appears in the work of Klartag and Livshyts
\cite{KlartagLivshytsRounding}, Livshyts \cite{LivshytsRounding},
and Livshyts, Tikhomirov and Vershynin \cite{LTV}.
The following version follows the argument of Lemma~5.3
in \cite{Tik} for the fixed finite support $S$.

\begin{lemma}\label{lem:rounding}
There are constants $c_{\ref{lem:rounding}}>0$ and
$C_{\ref{lem:rounding}}\ge1$ depending only on $S$ with the
following property.
Let $n\ge1$, $v\in\R^n$ and $\kappa>0$, and suppose that
\begin{align*}
\PP_\xi\left(|\langle\boldsymbol\xi,v\rangle|\le a\right)
 \le \frac{\kappa a}{\sqrt n}
 \qquad(a\ge\sqrt n).
\end{align*}
Then there is $y\in\Z^n$ such that:
\begin{enumerate}[label=(\roman*),leftmargin=2em,itemsep=2pt]
\item $\norm{y-v}_\infty\le1$, $|\langle\one_n,y-v\rangle| \le C_{\ref{lem:rounding}}\sqrt n$;

\item for every $a\ge\sqrt n$, $\PP_\xi\left(|\langle\boldsymbol\xi,y\rangle|\le a\right)
 \le \frac{C_{\ref{lem:rounding}}\kappa a}{\sqrt n}$;

\item the concentration function satisfies $\Lc_\xi(\langle\boldsymbol\xi,y\rangle,\sqrt n)
 \ge c_{\ref{lem:rounding}}
       \Lc_\xi(\langle\boldsymbol\xi,v\rangle,\sqrt n)$.
\end{enumerate}
\end{lemma}

\begin{proof}
Independently of $\boldsymbol\xi$, round the coordinates of $v$
independently to adjacent integers, preserving their expectations.
Let $\widetilde y$ be the resulting random vector and
$e=\widetilde y-v$.
Write $\PP_{\mathrm r},\EE_{\mathrm r}$ for probability and
expectation over the rounding.
The coordinates of $e$ are independent and centered, with
$|e_i|\le1$ and $\EE_{\mathrm r}e_i^2\le1/4$.
For $R=\max\{1,\max_{b\in S}|b|\}$ and every $b\in S^n$,
\begin{align*}
\EE_{\mathrm r}|\langle b,e\rangle|^2\le R^2n/4,
\qquad
\EE_{\mathrm r}|\langle\one_n,e\rangle|^2\le n/4.
\end{align*}

Let $h=\sqrt n$, $X=\langle\boldsymbol\xi,v\rangle$ and
$Y=\langle\boldsymbol\xi,\widetilde y\rangle$.
For $a\ge h$, define
\begin{align*}
F_a=\PP_\xi\left(|Y|\le a,\ |X|>2a\right),
\end{align*}
where the rounding is held fixed.
On $\{2^ja<|X|\le2^{j+1}a,\ |Y|\le a\}$, we have
$|Y-X|>2^{j-1}a$.
Chebyshev's inequality therefore gives
\begin{align*}
\EE_{\mathrm r}F_a
 &\le\sum_{j\ge1}\frac{R^2n}{2^{2j}a^2}\PP_\xi\left(|X|\le2^{j+1}a\right)\le\frac{2R^2\kappa h}{a}.
\end{align*}
For a sufficiently large $A=A(S)$, Markov's inequality yields
\begin{align*}
\PP_{\mathrm r}\left(
 F_{2^jh}>A\kappa2^j\text{ for some }j\ge0
\right)
 \le\frac{2R^2}{A}\sum_{j\ge0}4^{-j}
 \le\frac14.
\end{align*}
On the complementary event,
$\PP_\xi(|Y|\le2^jh)\le(A+2)\kappa2^j$ for every $j\ge0$.
Monotonicity of the small ball probability then gives
\begin{align*}
\PP_\xi(|Y|\le a)\le C\kappa a/h
\qquad(a\ge h),
\end{align*}
where $C$ depends only on $S$.

Let $p=\Lc_\xi(X,h)$, and choose $w\in\R$ such that
$\PP_\xi(|X-w|\le h)=p$.
For an integer $B\ge\max\{1,2R\}$, the variance bound gives
\begin{align*}
\EE_{\mathrm r}\PP_\xi\left(
 |X-w|\le h,\ |Y-X|>Bh
\right)
 \le\frac{R^2p}{4B^2}\le\frac p{16}.
\end{align*}
By Markov's inequality, with probability at least $7/8$,
\begin{align*}
\PP_\xi\left(|Y-w|\le(B+1)h\right)\ge p/2.
\end{align*}
Covering this interval by $B+1$ intervals of radius $h$ gives
$\Lc_\xi(Y,h)\ge p/[2(B+1)]$.
Finally,
$\PP_{\mathrm r}(|\langle\one_n,e\rangle|>h)\le1/4$.
These three events have a common realization.
Choosing $y$ to be such a realization proves the lemma.
\end{proof}

For $L\ge1$ and a unit vector $x$, define
\begin{align*}
\tau(x)=\sup\left\{
u\in(0,1]:
\Lc_\xi(\langle\boldsymbol\xi,x\rangle,u)>Lu
\right\}.
\end{align*}
We combine Lemma~\ref{lem:rounding} with
Theorem~\ref{thm:averaging} to construct a small set of
approximations at each concentration scale.

\begin{lemma}\label{lem:roundednet}
Fix $\delta\in(0,1/4]$ and $\nu\in(0,1/2]$.
There are constants
$L_{\ref{lem:roundednet}},C_{\ref{lem:roundednet}}\ge1$
depending only on $S,\delta,\nu$ with the following property.
Let $L=L_{\ref{lem:roundednet}}$.
For every $M\ge1$, $D>1$ and
$n\ge n_{\ref{lem:roundednet}}(S,\delta,\nu,M,D)$, every
$x\in\Incomp(\delta,\nu)$ satisfies
$q^{-n}/L\le\tau(x)\le\nu/\sqrt n$ and
\begin{align}
\Lc_\xi(\langle\boldsymbol\xi,x\rangle,u)
 \le Lu\qquad(u \ge \tau(x)), \quad
\Lc_\xi(\langle\boldsymbol\xi,x\rangle,\tau(x))
 = L\tau(x).
\label{eq:tauendpoint}
\end{align}

Suppose that $0<T\le\nu/4$ and
$N=\lfloor\nu/T\rfloor-1\le q^n(n+1)^{-D}$.
There is a deterministic set $\mathcal N_T\subset\Z^n$
with the following properties.
\begin{enumerate}[label=(\roman*),leftmargin=2em,itemsep=2pt]
\item Its cardinality satisfies
\begin{align}\label{eq:roundednet-size}
|\mathcal N_T|
 \le e^{-Mn}(C_{\ref{lem:roundednet}}N)^n.
\end{align}

\item For every $y\in\mathcal N_T$ and $a\ge\sqrt n$,
\begin{align}\label{eq:roundednet}
\PP_\xi\left(|\langle\boldsymbol\xi,y\rangle|\le a\right)
 \le \frac{C_{\ref{lem:roundednet}}Ta}{\sqrt n}.
\end{align}

\item Every $x\in\Incomp(\delta,\nu)$ with
$T/2<\tau(x)\le T$ admits $y\in\mathcal N_T$ such that,
with $e=y-\sqrt n\,x/\tau(x)$,
\begin{align}\label{eq:roundingerror}
\norm e_\infty\le1,
\qquad
|\langle\one_n,e\rangle|
 \le C_{\ref{lem:roundednet}}\sqrt n.
\end{align}
\end{enumerate}
\end{lemma}

\begin{proof}
For $x\in\Incomp(\delta,\nu)$, at least
$\lfloor\delta n\rfloor$ coordinates have absolute value greater
than $\nu/\sqrt n$, as in Lemma~5.1 of \cite{Tik}.
Lemma~\ref{lem:scalar} and interval covering give
\begin{align*}
\Lc_\xi(\langle\boldsymbol\xi,x\rangle,u)
 \le C(n^{-1/2}+u)\qquad(u\ge0),
\end{align*}
where $C$ depends only on $S,\delta,\nu$.
Thus $\tau(x)\le\nu/\sqrt n$ whenever
$L\ge C(1+\nu^{-1})$.
An atom of mass at least $q^{-n}$ gives
$\tau(x)\ge q^{-n}/L$.
Since $\tau(x)<1$, the definition of $\tau$ and the monotonicity
of the concentration function give \eqref{eq:tauendpoint}.
We impose a further lower bound on $L$ below.

Fix $x\in\Incomp(\delta,\nu)$ with $T/2<\tau(x)\le T$, and let
\begin{align*}
s=\tau(x),\qquad v=\frac{\sqrt n}{s}x.
\end{align*}
By \eqref{eq:tauendpoint}, for every $a\ge\sqrt n$,
\begin{align*}
\PP_\xi\left(|\langle\boldsymbol\xi,v\rangle|\le a\right)
 \le Lsa/\sqrt n.
\end{align*}
Moreover,
$\Lc_\xi(\langle\boldsymbol\xi,v\rangle,\sqrt n)\ge Ls$.
Apply Lemma~\ref{lem:rounding} with $\kappa=Ls$.
The resulting vector $y\in\Z^n$ satisfies
\begin{align*}
\norm{y-v}_\infty\le1,
\qquad
|\langle\one_n,y-v\rangle|\le C\sqrt n.
\end{align*}
For every $a\ge\sqrt n$, it also satisfies
\begin{align*}
\PP_\xi\left(|\langle\boldsymbol\xi,y\rangle|\le a\right)
 \le CLTa/\sqrt n,
\end{align*}
and its concentration function obeys
\begin{align*}
\Lc_\xi(\langle\boldsymbol\xi,y\rangle,\sqrt n)
 \ge cLs>cLT/2.
\end{align*}
Here $c,C$ depend only on $S$.

The geometric construction in Lemma~5.4 of \cite{Tik},
together with Corollary~5.5 and Lemma~5.6 there, gives a product
set $\mathcal A$ admissible with parameters $(N,n,K,\delta)$,
where $K=K(\delta,\nu)$, and a family $\Pi$ of coordinate
permutations with $|\Pi|\le C_0^n$ such that
\begin{align*}
y\in\bigcup_{\pi\in\Pi}\pi\mathcal A.
\end{align*}
Here $C_0$ is universal, and $\mathcal A,\Pi$ depend only on
$n,\delta,\nu,T$.
That construction uses only incompressibility and
$\norm{y-\sqrt n\,x/s}_\infty\le1$ for $T/2<s\le T$,
so it applies to the present approximations.

Let $L'
=L_{\ref{thm:averaging}}(S,\delta,K)$.
Choose $L=L_{\ref{lem:roundednet}}$ large enough for the initial
bound on $\tau$ and so that $L\ge4L'/(c\nu)$.
Since $\nu/2\le NT\le\nu$, every approximation above satisfies
\begin{align*}
\Lc_\xi(\langle\boldsymbol\xi,y\rangle,\sqrt n)
 >\frac{cLT}{2}
 \ge\frac{cL\nu}{4N}
 \ge\frac{L'}{N}.
\end{align*}
Define $\mathcal N_T$ to consist of all
$y\in\bigcup_{\pi\in\Pi}\pi\mathcal A$ satisfying this
concentration lower bound and
\begin{align*}
\PP_\xi\left(|\langle\boldsymbol\xi,y\rangle|\le a\right)
 \le CLTa/\sqrt n
 \qquad(a\ge\sqrt n).
\end{align*}
This set contains an approximation of every $x$ under
consideration.
Since the law of $\boldsymbol\xi$ is invariant under coordinate
permutations, Theorem~\ref{thm:averaging} gives
\begin{align*}
|\mathcal N_T|
 \le|\Pi|e^{-Mn}|\mathcal A|
 \le e^{-Mn}(C_0KN)^n.
\end{align*}
Choose $C_{\ref{lem:roundednet}}\ge\max\{C_0K,CL,C,1\}$.
Both constants in the statement depend only on $S,\delta,\nu$.
Taking $n_{\ref{lem:roundednet}}(S,\delta,\nu,M,D)$ large enough
to apply the geometric construction and
Theorem~\ref{thm:averaging} completes the proof.
\end{proof}

\begin{proof}[Proof of Theorem~\ref{thm:main}]
Fix $\delta,\nu$ as in Lemma~\ref{lem:structured}, and let
$L=L_{\ref{lem:roundednet}}$, $D=1+\varepsilon$ and
$t_n=(n+1)^Dq^{-n}$.
Assume first that $n\ge n_{\ref{thm:main}}(S,\varepsilon)$,
where the threshold will be chosen below.
For each $j\in[n]$, write $C_j=\operatorname{col}_j(M_n)$
and choose a unit vector $Y_j$ orthogonal to the other columns,
measurably with respect to those columns.
We first show that
\begin{align*}
\PP\left(
Y_j\in\Incomp(\delta,\nu),\ \tau(Y_j)>t_n
\right)\le q^{-2n}.
\end{align*}

Fix $j$, let $M_{-j}$ be the matrix obtained by deleting column
$j$, and set
$\mu=\EE\xi$ and
$Z_{-j}=M_{-j}-\mu\one_n\one_{n-1}^T$.
By Lemma~3.4 in \cite{Tik}, applied to the centered square matrix,
there is $K=K(S)\ge1$ such that the event
$\mathcal G_j=\{\norm{Z_{-j}}\le K\sqrt n\}$ satisfies
\begin{align*}
\PP(\mathcal G_j^c)\le e^{-(2\log q+1)n}.
\end{align*}
Take $n_{\ref{thm:main}}(S,\varepsilon)$ large enough that
$n\ge64$ and $t_n<\nu/\sqrt n$, and define
\begin{align*}
k=\left\lceil\log_2\frac{\nu}{\sqrt n\,t_n}\right\rceil,
\qquad
T_\ell=2^\ell t_n\quad(1\le\ell\le k).
\end{align*}
Since $\tau(x)\le\nu/\sqrt n$ for
$x\in\Incomp(\delta,\nu)$, the sets
\begin{align*}
\mathcal D_T=
\left\{x\in\Incomp(\delta,\nu):T/2<\tau(x)\le T\right\},
\qquad T\in\{T_1,\ldots,T_k\},
\end{align*}
cover all incompressible vectors with $\tau(x)>t_n$.
Each such scale satisfies $T<2\nu/\sqrt n\le\nu/4$ and
\begin{align*}
1\le N:=\lfloor\nu/T\rfloor-1
 \le\frac{\nu}{2t_n}
 \le q^n(n+1)^{-D}.
\end{align*}

Fix one of these scales and take $\mathcal N_T$ from
Lemma~\ref{lem:roundednet}, with exponent $M$ to be chosen below.
If $Y_j\in\mathcal D_T$ and $\mathcal G_j$ occurs, let
$v=\sqrt n\,Y_j/\tau(Y_j)$ and choose
$y\in\mathcal N_T$ as in \eqref{eq:roundingerror}.
Since $M_{-j}^Tv=0$, we have
\begin{align*}
\norm{M_{-j}^Ty}_2 \le K\sqrt n\,\norm{y-v}_2
 +|\mu|\sqrt{n-1}\,|\langle\one_n,y-v\rangle| \le K_1n,
\end{align*}
where $K_1\ge1$ depends only on $S,\delta,\nu$.
For each fixed $y\in\mathcal N_T$, the coordinates of
$M_{-j}^Ty$ are independent.
Thus \eqref{eq:roundednet} and the tensorization estimate
in Lemma~3.2 in \cite{Tik} give
\begin{align*}
\PP\left(\norm{M_{-j}^Ty}_2\le K_1n\right)
 \le (CT)^{n-1}.
\end{align*}
The union bound, \eqref{eq:roundednet-size} and $NT\le\nu$
therefore yield
\begin{align*}
\PP\left(Y_j\in\mathcal D_T,\ \mathcal G_j\right) \le e^{-Mn}(C_{\ref{lem:roundednet}}N)^n(CT)^{n-1} \le T^{-1}e^{-(M-C)n},
\end{align*}
where the constants depend only on $S,\delta,\nu$ and are
independent of $M,D$.

Choose $M\ge\max\{1,C+3\log q+1\}$, and take
$n_{\ref{thm:main}}(S,\varepsilon)$ large enough to apply
Lemma~\ref{lem:roundednet} with these parameters.
Since
\begin{align*}
\sum_{\ell=1}^kT_\ell^{-1}<t_n^{-1}\le q^n,
\end{align*}
we obtain
\begin{align*}
\PP\left(
Y_j\in\Incomp(\delta,\nu),\ \tau(Y_j)>t_n
\right) \le \PP(\mathcal G_j^c)
 +e^{-(M-C)n}\sum_{\ell=1}^kT_\ell^{-1}
 \le q^{-2n}.
\end{align*}

Conditional on $\mathcal F_j=\sigma(C_i:i\ne j)$,
the column $C_j$ retains its product distribution.
On the event
$\{Y_j\in\Incomp(\delta,\nu),\ \tau(Y_j)\le t_n\}$,
\eqref{eq:tauendpoint} therefore bounds the conditional
probability of $|\langle C_j,Y_j\rangle|\le h$ by
$L\max\{h,t_n\}$.
Together with the preceding estimate, this gives
\begin{align}
\PP\left(
Y_j\in\Incomp(\delta,\nu),\
|\langle C_j,Y_j\rangle|\le h
\right)
\le Lh+Lt_n+q^{-2n}
\qquad(h\ge0).
\label{eq:normal}
\end{align}

For $0\le z\le\nu$, let $\mathcal C_z$ be the event that
\begin{align*}
\inf_{x\in\Comp(\delta,\nu)}\norm{M_nx}_2
 \le\frac z{\sqrt n}
\qquad\text{or}\qquad
\inf_{x\in\Comp(\delta,\nu)}\norm{M_n^Tx}_2
 \le\frac z\nu.
\end{align*}
Applying \eqref{eq:structured} to $M_n$ and $M_n^T$ gives
\begin{align*}
\PP(\mathcal C_z)
\le a_n(S)
 +\left(\frac z{\sqrt n}+\frac z\nu+2q^{-n}\right)
 e^{-\eta_{\ref{lem:structured}}n}.
\end{align*}
On $\mathcal C_z^c$, the inequality
$|\langle C_j,Y_j\rangle|\le z/\nu$ implies
$Y_j\in\Incomp(\delta,\nu)$, since
$\norm{M_n^TY_j}_2=|\langle C_j,Y_j\rangle|$.
The averaging argument in the proof of Lemma~3.5 in \cite{RV}, applied on $\mathcal C_z^c$,
and \eqref{eq:normal} therefore give, for $n\ge2/\delta$,
\begin{align*}
\PP\left(s_n(M_n)\le z/\sqrt n,\ \mathcal C_z^c\right)
&\le\frac{2}{\delta n}\sum_{j=1}^n
 \PP\left(
 Y_j\in\Incomp(\delta,\nu),\
 |\langle C_j,Y_j\rangle|\le z/\nu
 \right)\\
&\le\frac2\delta
 \left(\frac{Lz}{\nu}+Lt_n+q^{-2n}\right).
\end{align*}
Combining the last two estimates and using
$(n+1)^D\le2^Dn^{1+\varepsilon}$ proves \eqref{eq:main}
for $0\le z\le\nu$ and
$n\ge n_{\ref{thm:main}}(S,\varepsilon)$.
Here the coefficient of $z$ depends only on $S$.
Choose the dimension threshold to include all preceding
restrictions and that of Lemma~\ref{lem:structured}.
Increasing $C_{\ref{thm:main}}$ covers $z>\nu$, and increasing
$C'_{\ref{thm:main}}$ covers the finitely many remaining dimensions.
\end{proof}

\begin{proof}[Proof of Corollary~\ref{cor:singularity}]
Fix $0<\varepsilon<1$. Let $\mathcal E_n$ be the union of the
events that $M_n$ has a zero row or column, or two equal or
opposite rows or columns. The individual probabilities sum
to $a_n(S)$, as defined in \eqref{eq:an}.

The row constraints have pairwise nonproportional coefficient
vectors $e_i,e_i\pm e_j$ ($i<j$). Two distinct row constraints
therefore have rank $2n$ on the entry array, as do two distinct
column constraints. A row constraint with coefficient vector $a$
and a column constraint with coefficient vector $b$ together
have rank $2n-1$, since
\begin{align*}
(\spn(a)\otimes\R^n)\cap(\R^n\otimes\spn(b))
=\spn(a\otimes b).
\end{align*}
By \eqref{eq:projectioncount} with $F=S$, a system of rank $h$
holds for at most $q^{n^2-h}$ of the $q^{n^2}$ entry arrays.
Each pairwise intersection thus has probability at most
$q^{1-2n}$. There are $O(n^4)$ pairs, so the union bound and
the first Bonferroni lower bound give
\begin{align*}
0\le a_n(S)-\PP\left(\mathcal E_n\right)
\le Cn^4q^{-2n},
\end{align*}
where $C$ depends only on $S$.

Since $\mathcal E_n$ implies singularity, this estimate and
\eqref{eq:main} at $z=0$ yield
\begin{align*}
a_n(S)-Cn^4q^{-2n} \le \PP\left(M_n\text{ is singular}\right) \le a_n(S)+C'n^{1+\varepsilon}q^{-n}.
\end{align*}
The first assertion follows, since
$n^4q^{-2n}=O_{S,\varepsilon}(n^{1+\varepsilon}q^{-n})$
and $a_n(S)=\Theta_S(n^2q^{-n})$.

For the two special cases, \eqref{eq:an} gives
\begin{align*}
a_n(\{-1,1\}) =2n(n-1)2^{-n}, \quad a_n(\{0,1\}) =(n^2+n)2^{-n}+n(n-1)4^{-n}.
\end{align*}
These identities give the remaining assertions.
\end{proof}

\section*{Acknowledgments}
The authors thank Professor Hanchao Wang for his guidance and
helpful suggestions.

\par\medskip
\noindent\textbf{Statement on AI use.}
The authors used ChatGPT with the GPT-6 Astra model in developing
the arguments of this paper. Before this use, the authors had
developed a Fourier argument using a moment of fixed order and
an earlier version of the set $V(\theta)$. By estimating the
contribution of approximate relations directly, this argument
gave an upper bound of
$2^{-n}\exp(o(\sqrt{n/\log n}))$
for the singularity probability of random sign matrices.

Seeking a sharp asymptotic, the authors considered arithmetic
and geometric approaches to controlling the relation sets and
the corresponding frequency regions. They asked ChatGPT to
explore these directions. In these discussions, the authors
and ChatGPT jointly developed the classification by dimension
and density, the use of moments whose order grows with $n$,
and separate estimates for bounded and large density loss.
The resulting method, including a weighted covering argument
for large density loss, yielded the sharp asymptotic for
random sign matrices.

The authors then asked whether the argument could be extended
to arbitrary finite real supports. ChatGPT proposed an approach
that avoids assuming that the support consists of algebraic
numbers. The main ingredients of this extension appear in
Lemma~\ref{lem:relations} and the proof of
Proposition~\ref{prop:far}.
The authors checked, organized, and rewrote the final proofs.
ChatGPT was also used for language editing.
The authors take full responsibility for the paper.

\appendix
\section{Determinant bounds for algebraic supports}\label{sec:algebraic}
\begin{lemma}\label{lem:fieldnorm}
For every fixed finite set $F$ of real algebraic numbers there is
$C_{\ref{lem:fieldnorm}}>0$ depending only on $F$ such that every invertible
$m\times m$ matrix $B$ with entries in $F$ satisfies
\begin{align*}
|\det B|\ge\exp\{-C_{\ref{lem:fieldnorm}}m\log(m+1)\}\qquad(m\ge1).
\end{align*}
\end{lemma}

\begin{proof}
The field $K=\mathbb Q(F)$ is a number field of degree $e$.
Choose a positive integer $H$ such that $Ha$ is an algebraic integer
for every $a\in F$. Then $\alpha=H^m\det B$ is a nonzero algebraic
integer in $K$. Its field norm is a nonzero integer, so
\begin{align*}
1\le |N_{K/\mathbb Q}(\alpha)|
     =\prod_{\sigma:K\hookrightarrow\mathbb C}|\sigma(\alpha)|.
\end{align*}
Set $B_0=\max(1,|\sigma(a)|:a\in F,\ \sigma:K\hookrightarrow\mathbb C)$.
For each embedding, Hadamard's inequality gives
\begin{align*}
|\sigma(\alpha)|
 =|\det(H\sigma(B))|\le(HB_0\sqrt m)^m.
\end{align*}
Separate the identity embedding in the norm product. It follows that
\begin{align*}
|\det B|=H^{-m}|\alpha|
 \ge H^{-m}(HB_0\sqrt m)^{-m(e-1)}.
\end{align*}
Taking logarithms proves the result, with a constant depending only
on $F$ and with $m=1$ covered by the same choice.
\end{proof}

\printbibliography
\end{document}